\documentclass[11pt,a4paper]{article}
\usepackage[margin=1.25in]{geometry}
\usepackage{amsmath, amsfonts, amssymb, wasysym, graphicx}
\usepackage{hyperref}
\hypersetup{colorlinks,citecolor=blue,linkcolor=blue}

\title{Splittings of one-ended groups with one-ended halfspaces}

\author{Michael Mihalik\footnote{Department of Mathematics, 1326 Stevenson Center, Vanderbilt University, Nashville, TN 37240, United States, michael.l.mihalik@vanderbilt.edu} and Sam Shepherd\footnote{Institut f\"ur Mathematische Logik und Grundlagenforschung, Fachbereich Mathematik und Informatik, Universit\"at M\"unster, Einsteinstra{\ss}e 62, 48149 M\"unster, Germany, sam.shepherd@uni-muenster.de} }

\usepackage{amsthm}

\newtheorem{theorem}{Theorem}[section]
\newtheorem{proposition}[theorem]{Proposition}
\newtheorem{lemma}[theorem]{Lemma}
\newtheorem{corollary}[theorem]{Corollary}

\theoremstyle{definition}
\newtheorem{remark}[theorem]{Remark}
\newtheorem{example}[theorem]{Example}
\newtheorem{notation}[theorem]{Notation}

\newcounter{definitionnum}
\newenvironment{definition}{\addvspace{12pt}\refstepcounter{definitionnum}
\noindent{\bf Definition \arabic{definitionnum}.}}{\par\addvspace{12pt}}

\usepackage{enumitem}
\setlist[enumerate,1]{label=(\arabic*)}

\usepackage{mathtools}

\usepackage{tikz-cd} 

\usepackage{tikz}
\tikzset{every loop/.style={}}
\usetikzlibrary{arrows,shapes,decorations.markings,automata,backgrounds,petri,bending,calc}

\usepackage{float}

\newcommand{\cA}{\mathcal{A}}

\newcommand{\cH}{\mathcal{H}}

\newcommand{\cM}{\mathcal{M}}

\newcommand{\cP}{\mathcal{P}}

\newcommand{\R}{\mathbb{R}}
\newcommand{\Z}{\mathbb{Z}}

\newcommand{\fa}{\mathfrak{a}}
\newcommand{\fb}{\mathfrak{b}}
\newcommand{\fc}{\mathfrak{c}}
\newcommand{\h}{\mathfrak{h}}

\newcommand{\hh}{\hat{\mathfrak{h}}}

\newcommand{\acts}{\curvearrowright}

\newcommand{\Cay}{\operatorname{Cay}}
\newcommand{\Aut}{\operatorname{Aut}}

\date{\today}
\begin{document}
\maketitle

\begin{abstract} 
	We introduce the notion of halfspaces associated to a group splitting, and investigate the relationship between the coarse geometry of the halfspaces and the coarse geometry of the group.
	Roughly speaking, the halfspaces of a group splitting are subgraphs of the Cayley graph obtained by pulling back the halfspaces of the Bass--Serre tree.
	Our first theorem shows that (under mild conditions) any splitting of a one-ended group can be upgraded to a splitting where all the halfspaces are one-ended.
	Our second theorem demonstrates that a one-ended group usually has a JSJ splitting where all the halfspaces are one-ended.
	And our third theorem states that if a one-ended finitely presented group $G$ admits a splitting such that some edge stabilizer has more than one end, but the halfspaces associated to the edge stabilizer are one-ended, then $H^2(G,\mathbb ZG)\ne \{0\}$; in particular $G$ is not simply connected at infinity and $G$ is not an $n$-dimensional duality group for $n\geq3$.

\end{abstract}

\section{Introduction}\label{Intro}

In 1982, B. Jackson \cite{J82b} proved that amalgamated products and HNN-extensions of one-ended finitely presented groups over finitely generated groups with more than one end are not simply connected at infinity. When the base groups are not one-ended, little was known. Exotic one-ended groups can be obtained as amalgamated products (HNN-extensions) with infinite-ended base groups and infinite-ended edge groups. These groups may or may not be simply connected at infinity. 
In this paper, instead of considering splittings with one-ended vertex groups we consider splittings with one-ended \emph{halfspaces}.
Given a finitely generated group $G$ and a splitting $G\acts T$ (here $T$ denotes the Bass--Serre tree) with finitely generated edge stabilizers, halfspaces in $T$ induce halfspaces in the Cayley graph of $G$ (see Section \ref{sec:halfspaces} for precise definitions).
Each halfspace is a connected subgraph of the Cayley graph, and its quasi-isometry type is independent of the choice of finite generating set for $G$ (Lemma \ref{lem:independent}).
If $G$ is one-ended one may expect that the halfspaces are also one-ended, but this is not necessarily the case -- see Example \ref{Ex1}.
However, our first theorem shows that under mild conditions one can upgrade a splitting of a one-ended group to one with one-ended halfspaces. 
This result is in analogy with results about one-ended halfspaces for cubulated groups that appear in \cite{Shepherd22}.

\begin{theorem}\label{thm:oneendedhs}
	Let $G\acts T$ be a non-trivial splitting with $G$ one-ended and finitely generated. Suppose the edge stabilizers are finitely generated and accessible.
	Then there is a non-trivial splitting $G\acts T'$ with minimal action such that:
	\begin{enumerate}
		\item The halfspaces of $G\acts T'$ are one-ended.
		\item Edge stabilizers (resp. vertex stabilizers) of $T'$ are finitely generated and are subgroups of the edge stabilizers (resp. vertex stabilizers) of $T$.
		\item For each edge $e'$ in $T'$ there exists an edge $e$ in $T$ such that $G_{e'}$ is a vertex stabilizer in some finite splitting of $G_e$ over finite subgroups.
	\end{enumerate}
\end{theorem}

As a corollary of Theorem \ref{thm:smallerstabs} (which is the main technical result behind Theorem \ref{thm:oneendedhs}) we also have the following result.

\begin{corollary}\label{cor:two-ended}
	If a one-ended group splits non-trivially over a two-ended group then the halfspaces of this splitting are one-ended.
\end{corollary}

We remark that, in the finitely presented case, Corollary \ref{cor:two-ended} can alternatively be deduced from a result of Papasoglu \cite[Definition 1.4 and Lemma 1.8]{Papasoglu05} (note that Lemma 1.8 in \cite{Papasoglu05} is not quite stated correctly: the group $G$ should be one-ended).

As an elementary example of Corollary \ref{cor:two-ended} one can consider an essential simple closed curve in a closed hyperbolic surface $S$ and the associated splitting of $\pi_1(S)$ over $\Z$ with free vertex groups; the lift of the curve to the universal cover $\mathbb{H}^2$ splits $\mathbb{H}^2$ into two halfspaces each of which is a half-plane. 

There is a long history in the literature of studying group splittings and of finding ways to upgrade a given splitting to one that has better properties or that carries more information about the group.
Probably the most well known example of this is the notion of a \emph{JSJ splitting} or \emph{JSJ tree}, due to Rips--Sela \cite{RipsSela97} and Dunwoody--Sageev \cite{DunwoodySageev99}.
In its most general form, if we have a group $G$ and a family $\cA$ of subgroups of $G$ that is closed under conjugating and taking subgroups, then there is a notion of \emph{JSJ tree of $G$ over $\cA$} due to Guirardel--Levitt \cite{GuirardelLevitt17} (see Definition \ref{defn:JSJ}).
Our next theorem shows that, under mild conditions on $\cA$, any one-ended group has a JSJ splitting with one-ended halfspaces.

\begin{theorem}\label{thm:JSJ}
	Let $G$ be a finitely generated one-ended group and let $\cA$ be a family of subgroups of $G$ that is closed under conjugating and taking subgroups.
	\begin{enumerate}
		\item If the groups in $\cA$ are finitely generated and accessible and there exists a JSJ tree $T$ of $G$ over $\cA$, then there exists a JSJ tree $T'$ of $G$ over $\cA$ such that the halfspaces of $G\acts T'$ are one-ended.
		\item If $T$ is a JSJ tree of $G$ over $\cA$ with finitely generated edge stabilizers, such that no vertex stabilizer of $T$ fixes an edge in $T$, then the halfspaces of $G\acts T$ are one-ended.
	\end{enumerate}
\end{theorem}

Although the definition of halfspaces for a group splitting is new, there are similar ideas that have been used many times in the literature, especially regarding the construction of JSJ trees.
For example, if $(G,\mathbb{P})$ is a relatively hyperbolic group with connected Bowditch boundary, then the construction of the exact cut pair/cut point tree $T$ for $G$ \cite{HaulmarkHruska23} (which is the same as the JSJ tree of cylinders for elementary splittings relative to $\mathbb{P}$) can be formulated in terms of halfspaces.
Indeed, the edges of $T$ correspond to decompositions of the Bowditch boundary $\partial(G,\mathbb{P})=U_1\cup U_2$ (which can be thought of as halfspaces in the boundary) such that $U_1\cap U_2$ is a cut point or an inseparable exact cut pair and $U_1 - U_2$ is one of the connected components of $\partial(G,\mathbb{P}) - U_1\cap U_2$.
Moreover, it can be deduced from \cite[Proposition 6.7]{HaulmarkHruska23} that if an edge $e$ in $T$ corresponds to a decomposition $\partial(G,\mathbb{P})=U_1\cup U_2$, then $U_1$ and $U_2$ are the limit sets of the halfspaces of $G\acts T$ associated to $e$ (see Definition \ref{defn:halfspaceT}).
Halfspaces are also fundamental in the theory of CAT(0) cube complexes and the study of group actions on CAT(0) cube complexes -- which can be regarded as a generalization of group splittings.

The next few results concern the simple connectivity at infinity and the second cohomology of groups, and their connections to one-ended halfspaces.
We start with a theorem that works with an analogue of one-ended halfspaces for CW-complexes.

\begin{theorem}\label{GeoSplit}
Suppose $X$ is a locally finite CW-complex, and $X_1$ and $X_2$ are connected one-ended subcomplexes of $X$ such that $X_1\cup X_2=X$. If $K$ is a finite subcomplex of $X$ (possibly empty) such that  $(X_1\cap X_2)-K$ has more than one unbounded component, then $X$ does not have pro-finite first homology at infinity. In particular, $X$ is not simply connected at infinity.
\end{theorem}

We note that if a finitely presented group $G$ is simply connected at infinity then $H^2(G,\mathbb ZG)=\{0\}$, and $H^2(G,\mathbb ZG)=\{0\}$ if and only if the first homology at infinity of $G$ is pro-finite (see Section \S\ref{sec:scatinfty} for the relevant definitions and results).

As a corollary of Theorem \ref{GeoSplit}, we obtain the following generalization of Jackson's result, working with one-ended halfspaces rather than one-ended vertex groups.
Given a splitting of a one-ended group, if all the vertex groups are one-ended then all the halfspaces are also one-ended (see Lemma \ref{lem:oneendedvertexgroup}). 
The details for how we deduce Corollary \ref{Acyclic} from Theorem \ref{GeoSplit} are given in Section \ref{sec:acyclicthm}.

\begin{corollary}\label{Acyclic}
Let $G\acts T$ be a non-trivial splitting with minimal action, with $G$ one-ended and finitely presented. Suppose the edge stabilizers are finitely generated, and suppose there is some edge stabilizer $G_e$ with more than one end.
If the two halfspaces of $G\acts T$ associated to $e$ are one-ended, then $H^2(G,\mathbb ZG)\ne \{0\}$. 
\end{corollary}

Combining Theorem \ref{thm:oneendedhs} with Corollary \ref{Acyclic}, we get a further corollary about splittings over virtually free groups.

\begin{corollary} If $G$ is a finitely presented one-ended group that splits non-trivially over a virtually free group, then $H^2(G,\mathbb ZG)\ne \{0\}$. 
\end{corollary}

By folding, one can transform a non-trivial splitting over a subgroup $C$ into a non-trivial splitting over a larger subgroup. We use this idea to produce an example (see Example \ref{Ex1}) of a finitely presented one-ended simply connected at infinity group $G$ and a non-trivial splitting of $G$ over an infinite-ended group. So, it is important to know when halfspaces are one-ended.

Perhaps the most important application of the notion of simple connectivity at infinity is to manifolds. The notion of simple connectivity at infinity was popularized by J. Stallings \cite{St62}, J. Munkres \cite{Mun60}, E. Luff \cite{Luf67} and L. Siebenmann \cite{Sie68} in the 1960's as a way of deciding when a contractible open $n$-manifold is homeomorphic to $\mathbb R^n$ for $n\geq 5$. Following the work of Freedman in dimension 4 (see Corollary1.2 of \cite{Fre82}) and Perelman in dimension 3 (see L. Husch and T. Price's paper \cite{HP70} and the addendum \cite{HP71}), these results were extended to all dimensions $n>2$. So for contractible open topological, PL and differential manifolds $M$ of dimension $n>2$ we have that $M$ is homeomorphic to $\mathbb R^n$ if and only if $M$ is simply connected at infinity. (Note that in dimension 4 there are exotic PL and differential structures on $\mathbb R^4$ - see S. Donaldson's paper \cite{Don83}.)  M. Davis \cite{Davis83} produced closed manifolds with infinite fundamental group in dimensions $n\geq 4$ whose universal covers were contractible, and he proved these universal covers are not simply connected at infinity (and hence not homeomorphic to $\mathbb R^n$). 
The fundamental group $G$ of a Davis manifold is a subgroup of finite index of a right angled Coxeter group $W$ and so many ``visual" splittings of $W$ (and $G$) are easily obtained. 
However, we have a restriction on splittings of such groups by the following corollary of Corollary \ref{Acyclic}.

\begin{corollary}
If $G$ is the fundamental group of a closed aspherical $n$-manifold (or more generally an $n$-dimensional duality group, see Section \ref{sec:examplesoneended}) with $n\geq3$, then there cannot be a splitting of $G$ over a group with more than one end where all the halfspaces are one-ended.
\end{corollary} 

The contrapositive of the corollary is also interesting: if a group $G$ of cohomological dimension $n\geq3$ splits over an infinite-ended group with all of the halfspaces one-ended, then $G$ is not a duality group.
We give some examples of such groups in Section \ref{sec:examplesoneended} (which are carefully chosen so that we cannot see another way of showing that these groups are not duality groups).

\S \ref{sec:halfspaces} contains three definitions for the halfspaces associated to a group splitting, along with a lemma to explain why they are equivalent (up to quasi-isometry).

\S \ref{sec:pocsets} contains background about pocsets, which is needed for the proofs of Theorems \ref{thm:oneendedhs} and \ref{thm:JSJ}.

\S \ref{sec:chopping} contains the proofs of Theorems \ref{thm:oneendedhs} and \ref{thm:JSJ}, both of which are deduced from Theorem \ref{thm:smallerstabs}. The latter theorem essentially says that if a splitting of a one-ended group contains a halfspace with more than one-end, then one can obtain a new splitting with smaller edge stabilizers. The basic idea behind the proof is to chop up the halfspace with more than one end.
From this we build a certain pocset, and the new splitting is obtained from the cubing of this pocset (which we show is a tree).

\S \ref{sec:scatinfty} begins with the definitions and connections between simple connectivity at infinity and various notions of first homology at infinity for topological spaces. Classical results show that these notions are preserved under any quasi-isometry or proper 2-equivalence between CW-complexes and can be extended to finitely presented groups. 

\S \ref{sec:acyclicthm} contains the proofs of Theorem \ref{GeoSplit} and Corollary \ref{Acyclic}. The proof of Theorem \ref{GeoSplit} is topological. We point out that it is a rare situation when $H^2(G,\mathbb ZG)$ is shown to be non-trivial by non-homological methods. 

\S \ref{sec:artificial} begins with a lemma that describes how to alter a given non-trivial amalgamated product to obtain a non-trivial one with larger edge and vertex groups. Example \ref{Ex1} begins with a finitely presented group $G$ that is simply connected at infinity and splits as an amalgamated product with simply connected at infinity vertex groups and a one-ended edge group. A second splitting is constructed for $G$ with an infinite-ended vertex group and an infinite-ended edge group. So simply connected at infinity groups can be split non-trivially over infinite-ended groups. Corollary \ref{Acyclic} implies that a halfspace for this last splitting is not one-ended (so halfspaces for one-ended simply connected at infinity groups need not always be one-ended). 

\S \ref{sec:examplesoneended} developes two general yet elementary constructions of HNN extensions and amalgams with one-ended halfspaces (Propositions \ref{prop:stablecommute} and \ref{prop:double}). We then use these constructions to produce examples of HNN extensions and amalgams that are not duality groups, but which cannot be distinguished from duality groups using previous literature (to the best of our knowledge).

\textbf{Acknowledgements:}\,
We thank Alex Margolis for suggesting Example \ref{ex:notqiemb}.
And we are grateful for the referee's careful reading of our paper and their helpful comments and corrections.

\section{Group splittings and their halfspaces}\label{sec:halfspaces}

\begin{definition}
	A \emph{splitting} of a group $G$ is an action $G\acts T$ on a tree without edge inversions.
	The action is \emph{minimal} if $T$ has no proper invariant subtree.
	The splitting is \emph{non-trivial} if $G\acts T$ has no fixed point and \emph{finite} if the action $G\acts T$ is cocompact.
	Note that any splitting of a finitely generated group with minimal action is finite.
\end{definition}

\begin{notation}
	Whenever we talk about a graph $X$ (which will often be a tree), we denote the vertex and edge sets by $VX$ and $EX$ respectively.
\end{notation}

We now give three definitions for the halfspaces associated to a splitting of a finitely generated group $G\acts T$ with minimal action and finitely generated edge stabilizers, along with a lemma to explain why they are equivalent.
A halfspace is always a connected subgraph of a larger graph, and is considered as a metric space with respect to its induced path metric.
The first two definitions of halfspace are shorter and more intuitive, while the third is more convenient for the proof of Theorem \ref{thm:oneendedhs}.
Before giving the definitions, we note the following lemma, which implies that vertex stabilizers are always finitely generated when we have a splitting with halfspaces.

\begin{lemma}\cite[Lemma 8.32]{Cohen89}\\\label{lem:fgvertex}
	Any finite splitting $G\acts T$ of a finitely generated group with finitely generated edge stabilizers has finitely generated vertex stabilizers.
\end{lemma}

\begin{definition} (First definition of halfspaces)\\\label{defn:halfspaceamal}
	Let $G=A*_C B$ be a non-trivial amalgam with $A, B$ and $C$ finitely generated.
	Given finite generating sets $S_A$ and $S_B$ for $A$ and $B$, we know that $S=S_A\cup S_B$ is a finite generating set for $G$.
	Let $\Cay(G,S)$ denote the Cayley graph of $G$ with respect to $S$.
	Define $\h_A$ (resp. $\h_B$) to be the induced subgraph of $\Cay(G,S)$ whose vertices are the elements of $G$ whose normal form starts with an element of $A$ (resp. an element of $B$).
	We refer to $\h_A, \h_B$ and all of their $G$-translates as \emph{halfspaces} of the splitting $G=A*_C B$.
	Note that these halfspaces are connected subgraphs of $\Cay(G,S)$, and $\h_A\cap \h_B$ is the induced subgraph of $C$.
	
	A similar definition can be made for HNN extensions, as follows.
	Suppose $G=A*_C$ is an HNN extension, with $A$ and $C$ finitely generated. Let $S_A$ be a finite generating set for $A$, and let $t$ be the stable letter of the HNN extension. Let $\Cay(G,S)$ denote the Cayley graph of $G$ with respect to $S=S_A\cup\{t\}$.
	Define the halfspace $\h^+$ to be the induced subgraph of $\Cay(G,S)$ whose vertices are elements of $G$ with normal forms that do not start with an element of $C$ followed by a negative power of $t$.
	Define the halfspace $\h^-$ to be the induced subgraph of $\Cay(G,S)$ whose vertices are either elements of $C$ or elements of $G$ with normal forms that do start with an element of $C$ followed by a negative power of $t$.
	We refer to $\h^+, \h^-$ and all of their $G$-translates as \emph{halfspaces} of the splitting $G=A*_C$.
	Note that these halfspaces are connected subgraphs of $\Cay(G,S)$, and $\h^+\cap \h^-$ is the induced subgraph of $C$.	
\end{definition}

\begin{remark}\label{remk:CayleyC}
	In Definition \ref{defn:halfspaceamal}, if one chooses the finite generating set $S$ to contain a generating set for $C$, then $\h_A\cap \h_B$ (or  $\h^+\cap \h^-$ in the HNN case) is connected and is isomorphic to a Cayley graph of $C$.
\end{remark}

The second and third definitions of halfspace require the notion of halfspaces in a tree.

\begin{definition}\label{defn:treehalfspaces}
	Let $T$ be a tree. Denote the midpoint of an edge $e$ by $\hat{e}$. Define a \emph{halfspace} of $T$ to be the union of $\hat{e}$ with one of the components of $T-\hat{e}$.
	Let $\cH(T)$ denote the set of halfspaces of $T$.
\end{definition}

\begin{definition} (Second definition of halfspaces)\\\label{defn:halfspaceT}
Let $G\acts T$ be a splitting of a finitely generated group $G$ with minimal action and finitely generated edge stabilizers. 
Fix an orbit map $f:G\to VT$. Let $S$ be a finite generating set for $G$ such that the induced subgraph of $f^{-1}(\h)$ in $\Cay(G,S)$ is connected for every halfspace $\h\in\cH(T)$ -- such $S$ exists by Lemma \ref{lem:Sexists} below.
These induced subgraphs are called the \emph{halfspaces} of $G\acts T$ (with respect to $S$).
If $\h$ is bounded by the edge midpoint $\hat{e}$ then we say that the halfspace of $G\acts T$ induced by $f^{-1}(\h)$ is a halfspace \emph{associated to $e$}.
\end{definition}

\begin{lemma}\label{lem:Sexists}
	Let $G\acts T$ be a splitting of a finitely generated group $G$ with minimal action and finitely generated edge stabilizers. 
	Fix an orbit map $f:G\to VT$. There exists a finite generating set $S$ for $G$ such that the induced subgraph of $f^{-1}(\h)$ in $\Cay(G,S)$ is connected for every halfspace $\h\in\cH(T)$.
\end{lemma}
\begin{proof}
Let $\h\in\cH(T)$ be a halfspace associated to an edge $e\in ET$.
We will find a finite generating set $S$ for $G$ such that the induced subgraph of $f^{-1}(\h)$ in $\Cay(G,S)$ is connected.
The lemma then follows by repeating the argument for a finite set of $G$-orbit representatives in $\cH(T)$, and taking the union of the corresponding finite generating sets for $G$.

Let $S_0$ be a finite symmetric generating set for $G$, and let
\begin{equation}\label{df1dfs}
	R=\max\{d(f(1),f(s))\mid s\in S_0\}=\max\{d(f(g),f(gs))\mid g\in G, s\in S_0\}.
\end{equation}
Let $\h^*\in\cH(T)$ be the halfspace which is complementary to $\h$, and define
\begin{equation}
	X_e=\{(g,gs)\in G^2\mid f(g)\in \h, f(gs)\in\h^*, s\in S_0\}.
\end{equation}
Since the orbit map $f$ is $G$-equivariant, the left action of $G_e$ on $G^2$ stabilizes $X_e$.
We claim that $X_e$ contains only finitely many $G_e$-orbits.
Indeed, suppose $s\in S_0$ and $(g_1,g_1s),(g_2,g_2s)\in X_e$ with $d(f(g_1),e)=d(f(g_2),e)$.
The element $g_2g_1^{-1}$ maps the geodesic in $T$ from $f(g_1)$ to $f(g_1s)$ onto the geodesic from $f(g_2)$ to $f(g_2s)$.
Since $d(f(g_1),e)=d(f(g_2),e)$, it must be that $g_2g_1^{-1}\in G_e$, so $(g_1,g_1s),(g_2,g_2s)\in X_e$ are in the same $G_e$-orbit.
For $(g,gs)\in X_e$, there are only finitely many possibilities for the choice of $s\in S_0$, and the distance $d(f(g),e)$ is bounded by $R$ from (\ref{df1dfs}) because $e$ lies on the geodesic in $T$ from $f(g)$ to $f(gs)$.
This proves the claim that $X_e$ contains only finitely many $G_e$-orbits.

Now let $\{(g_i,g_is_i)\}_{i=1}^n$ be a finite set of $G_e$-orbit representatives in $X_e$.
Let $S_e$ be a finite generating set for $G_e$,  and let 
\begin{equation}
	S_e'=\bigcup_{i=1}^ng_i^{-1}S_e g_i\quad\text{and}\quad S_e^\#=\{g_i^{-1}g_j\mid i,j\in\{1,\dots ,n\}\}.
\end{equation}
We claim that, for each $i=1,\dots,n$ and each $h\in G_e$, there is a path from $g_i$ to $hg_i$ in $\Cay(G,S_0\cup S_e')$ that stays in $f^{-1}(\h)$.
Indeed, write $h=s_1s_2\cdots s_k$ with $s_j\in S_e$.
Then the desired path is
\begin{align*}
&(g_i,g_i(g_i^{-1}s_1g_i),g_i(g_i^{-1}s_1g_i)(g_i^{-1}s_2g_i),\dots,g_i(g_i^{-1}s_1g_i)(g_i^{-1}s_2g_i)\cdots(g_i^{-1}s_kg_i))\\
= &(g_i,s_1g_i,s_1s_2g_i,\dots,hg_i).
\end{align*}
Let $S=S_0\cup S_e'\cup S_e^\#$.
It follows from the previous claim that, for any $(g,gs),(g',g's')\in X_e$, there is a path from $g$ to $g'$ in $\Cay(G,S)$ that stays in $f^{-1}(\h)$.

We now complete the proof of the lemma by showing that the induced subgraph of $f^{-1}(\h)$ in $\Cay(G,S)$ connected.
Indeed, let $g,g'\in f^{-1}(\h)$.
Take a path $\gamma$ in $\Cay(G,S_0)$ from $g$ to $g'$, and consider it as a path in $\Cay(G,S)$.
Each maximal subpath of $\gamma$ that lies in $f^{-1}(\h^*)$ must be preceded by an edge $(g'',g''s'')\in f^{-1}(\h)\times f^{-1}(\h^*)$ and immediately followed by an edge $(g''',g'''s''')\in f^{-1}(\h^*)\times f^{-1}(\h)$ (with $s'',s'''\in S_0$).
Since $S_0$ is symmetric, we have $(g'',g''s''),(g'''s''',g''')\in X_e$.
Hence it follows from the arguments above that the subpath of $\gamma$ from $g''$ to $g'''s'''$ can be replaced by a path in $f^{-1}(\h)$.
Doing this for every maximal subpath of $\gamma$ that lies in $f^{-1}(\h^*)$, we obtain a path from $g$ to $g'$ in $f^{-1}(\h)$, as required.
\end{proof}

\begin{definition} (Third definition of halfspaces)\\\label{defn:treeofspaces}
	Let $G\acts T$ be a splitting of a finitely generated group $G$ with minimal action and finitely generated edge stabilizers.
	A \emph{tree of spaces} for $G\acts T$ is a locally finite graph $X$ equipped with a geometric action of $G$, and a surjective $G$-equivariant map $p:X\to T$, such that:
	\begin{enumerate}
		\item $p$ is simplicial with respect to the first barycentric subdivision of $T$, and
		\item the $p$-preimages of vertices and edge midpoints in $T$ are connected.
	\end{enumerate} 
A tree of spaces can be constructed by modifying a Cayley graph of $G$, for instance apply \cite[Proposition 8.2]{MargolisShepherdStarkWoodhouse23} to $G$ with the discrete topology.
	Define a \emph{halfspace} of $X$ to be the preimage $p^{-1}(\h)$ of a halfspace $\h\in\cH(T)$.
Write $\cH(X)$ for the set of halfspaces of $X$.
	Note that halfspaces of $X$ are connected subgraphs of $X$.	
\end{definition}

In the following three lemmas we prove that these three definitions of halfspace are well-defined and equivalent up to quasi-isometry.
This is not as straightforward to prove as one might imagine, since the inclusion of a halfspace $\h\xhookrightarrow{}\Cay(G,S)$ (say for the second definition) is not necessarily a quasi-isometric embedding (see Example \ref{ex:notqiemb}).

\begin{lemma}\label{lem:preindependent}
Let $G\acts T$ be a splitting of a finitely generated group $G$ with minimal action and finitely generated edge stabilizers. 
Fix a finite generating set $S$ and an orbit map $f:G\to VT$. Let $p:X\to T$ be a tree of spaces as in Definition \ref{defn:treeofspaces}.
Fix an edge $e\in ET$ and let $\h\in\cH(T)$ be one of the halfspaces associated to $e$.
Suppose that the induced subgraph $\h_1$ of $f^{-1}(\h)$ in $\Cay(G,S)$ is connected (but make no assumption regarding the other halfspaces of $T$).
Let $\h_2=p^{-1}(\h)\subseteq X$ be the halfspace corresponding to $\h$ from Definition \ref{defn:treeofspaces}.
Let $i:G\to VX$ be an orbit map for the action of $G$ on $X$ such that the following diagram commutes.
\begin{equation}
	\begin{tikzcd}[
		ar symbol/.style = {draw=none,"#1" description,sloped},
		isomorphic/.style = {ar symbol={\cong}},
		equals/.style = {ar symbol={=}},
		subset/.style = {ar symbol={\subset}}
		]
		G\ar{dr}[swap]{f}\ar{r}{i}&X\ar{d}{p}\\
		&T
	\end{tikzcd}
\end{equation}
Then the map $i:(V\h_1,d_{\h_1})\to (V\h_2,d_{\h_2})$ is a quasi-isometry, where $d_{\h_1},d_{\h_2}$ denote the induced path metrics on $\h_1,\h_2$ respectively.
\end{lemma}
\begin{proof}
Denote the metrics on $\Cay(G,S)$ and $X$ by $d_S$ and $d_X$ respectively.
The Milnor--Svarc Lemma implies that $i:(G,d_S)\to(X,d_X)$ is a quasi-isometry. Say $C\geq1$ is such that
\begin{equation}\label{CCQI}
	C^{-1}d_S(g,g')-C\leq d_X(i(g),i(g'))\leq Cd_S(g,g')+C
\end{equation}
for all $g,g'\in G$, and such that the $C$-neighborhood of $i(G)$ is equal to the whole of $X$.

Let $\hh_2:=p^{-1}(\hat{e})\subseteq\h_2$.
By Definition \ref{defn:treeofspaces}, $G$ acts cocompactly on $X$, hence $G_e$ acts cocompactly on $\hh_2$.
As $X$ is locally finite, $G_e$ also acts cocompactly on the $4C$-neighborhood of $\hh_2$.
It follows that there is a constant $L$ such that, for any $v,v'\in V\h_2$, if $d_X(v,v')\leq 2C$ and $d_X(v,\hh_2)\leq2C$ then $d_{\h_2}(v,v')\leq L$.
We may also assume $L\geq2C$.
Now consider adjacent vertices $g,g'\in V\h_1$.
By (\ref{CCQI}), we have $d_X(i(g),i(g'))\leq Cd_S(g,g')+C= 2C$.
Note that any path leaving $\h_2$ must cross $\hh_2$, so if $d_X(i(g),\hh_2)\geq2C$ then $i(g')\in\h_2$ and $d_{\h_2}(i(g),i(g'))=d_X(i(g),i(g'))\leq2C$.
Otherwise, if $d_X(i(g),\hh_2)<2C$ then $d_{\h_2}(i(g),i(g'))\leq L$ by definition of $L$.
In either case, we have $d_{\h_2}(i(g),i(g'))\leq L$.
By mapping paths in $\h_1$ into $X$, we deduce that 
\begin{equation}\label{d2leq}
	d_{\h_2}(i(g),i(g'))\leq Ld_{\h_1}(g,g')
\end{equation}
for all $g,g'\in V\h_1$.

Since $G_e$ acts cocompactly on the $C$-neighborhood of $\hh_2$, and $G_e$ stabilizes $\h_1$, there is a constant $K$ such that, for any $v\in V\h_2$, if $d_X(v,\hh_2)\leq C$ then $d_{\h_2}(v,i(V\h_1))\leq K$.
We may also assume $K\geq C$.
We now claim that 
\begin{equation}\label{coarsesurj}
	d_{\h_2}(v,i(V\h_1))\leq K
\end{equation} 
for all $v\in V\h_2$.
Indeed, if $d_X(v,\hh_2)\leq C$ then the claim follows by definition of $K$.
Otherwise, if $d_X(v,\hh_2)>C$, then take $g\in G$ with $d_X(v,i(g))\leq C$ (which exists by definition of $C$).
The shortest path from $v$ to $i(g)$ does not cross $\hh_2$, so it stays within $\h_2$, which means $g\in\h_1$ and $d_{\h_1}(v,i(g))=d_X(v,i(g))\leq C\leq K$.

Put $D=2C^2(3K+1+C)+K$.
Using properness of the $G$-action on $X$, and the fact that $G_e$ acts cocompactly on the $(D+K)$-neighborhood of $\hh_2$, there is a constant $M$ such that, for any $g,g'\in V\h_1$, if $d_X(i(g),\hh_2)\leq D+K$ and $d_X(i(g),i(g'))\leq 2K+1$ then $d_{\h_1}(g,g')\leq M$.
And we may assume $M\geq D$.
Now take $g,g'\in V\h_1$.
Let $i(g)=v_0,v_1,\dots,v_n=i(g')$ be the vertices of a shortest path in $\h_2$ from $i(g)$ to $i(g')$.
For each $v_j$, take $g_j\in V\h_1$ with $d_{\h_1}(v_j,i(g_j))\leq K$ (which exists by (\ref{coarsesurj})). We may assume that $g_0=g$ and $g_n=g'$.
Now consider a pair of consecutive vertices $v_j, v_{j+1}$.
Note that
\begin{equation}\label{digj}
	d_X(i(g_j),i(g_{j+1}))\leq d_X(i(g_j),v_j)+d_X(v_j,v_{j+1})+d_X(v_{j+1},i(g_{j+1}))\leq2K+1.
\end{equation}
If one of $v_j, v_{j+1}$ is in the $D$-neighborhood of $\hh_2$, then one of $i(g_j),i(g_{j+1})$ is in the $(D+K)$-neighborhood of $\hh_2$, so $d_{\h_1}(g_j,g_{j+1})\leq M$ by (\ref{digj}) and definition of $M$.
Now suppose neither $v_j$ nor $v_{j+1}$ is in the $D$-neighborhood of $\hh_2$.
By (\ref{CCQI}) and (\ref{digj}) we have $d_S(g_j,g_{j+1})\leq C(2K+1+C)$.
Moreover, any shortest path $\gamma$ in $\Cay(G,S)$ from $g_j$ to $g_{j+1}$ will have $i$-image contained in the $2C^2(3K+1+C)$-neighborhood of $i(g_j)$ (again using (\ref{CCQI})), which is contained in $\h_2$ since $$d_X(i(g_j),\hh_2)\geq d_X(v_j,\hh_2)-d_X(v_j,i(g_j))> D-K=2C^2(3K+1+C).$$
So $\gamma\subseteq\h_1$ and 
$$d_{\h_1}(g_j,g_{j+1})=d_S(g_j,g_{j+1})\leq C(2K+1+C)<D\leq M.$$
In either case we get $d_{\h_1}(g_j,g_{j+1})\leq M$, hence
\begin{equation}\label{d1leq}
	d_{\h_1}(g,g')\leq Mn=Md_{\h_2}(i(g),i(g')).
\end{equation}

The combination of (\ref{d2leq}), (\ref{coarsesurj}) and (\ref{d1leq}) tells us that $i:(V\h_1,d_{\h_1})\to (V\h_2,d_{\h_2})$ is a quasi-isometry, as required.
\end{proof}

\begin{lemma}\label{lem:independent}
	The second and third definitions for halfspaces given above are coarsely well-defined and coarsely equivalent.
	More precisely, if we fix a splitting $G\acts T$ and a halfspace $\h$ of $T$, then the halfspace of $G\acts T$ corresponding to $\h$ remains the same up to quasi-isometry when we
	\begin{enumerate}
		\item\label{item:change1} switch between the two definitions of halfspace,
		\item\label{item:change2} change the finite generating set for $G$ or the orbit map $f$ from Definition \ref{defn:halfspaceT}, or
		\item\label{item:change3} change the graph $X$ from Definition \ref{defn:treeofspaces}.
	\end{enumerate}
	In particular, the notion of the halfspaces of a splitting $G\acts T$ being one-ended is well-defined.

\end{lemma}
\begin{proof}	
	Let $S$ be a finite generating set for $G$ as in Definition \ref{defn:halfspaceT} (so the induced subgraph of $f^{-1}(\h)$ in $\Cay(G,S)$ is connected for every halfspace $\h\in\cH(T)$).
	Let $\h\in\cH(T)$ be a halfspaces of $T$.
	Let $\h_1\subset\Cay(G,S)$ and $\h_2\subset X$ be the halfspaces corresponding to $\h$ from Definitions \ref{defn:halfspaceT} and \ref{defn:treeofspaces} respectively.	
	Lemma \ref{lem:preindependent} implies that $\h_1$ and $\h_2$ are quasi-isometric, which proves item \ref{item:change1}.
	Moreover, we can keep using Lemma \ref{lem:preindependent} to show that $\h_1$ and $\h_2$ remain the same up to quasi-isometry when making the changes indicated in \ref{item:change2} and \ref{item:change3}.
	For instance, if $S'$ is another finite generating set for $G$ as in Definition \ref{defn:halfspaceT}, and $\h'_1\subseteq\Cay(G,S')$ is the halfspace corresponding to $\h$ and $S'$ from Definition \ref{defn:halfspaceT}, then Lemma \ref{lem:preindependent} implies that the halfspaces $\h_1$ and $\h'_1$ are both quasi-isometric to $\h_2$, hence they are quasi-isometric to each other.
	Similarly, the halfspace $\h_1$ remains the same up to quasi-isometry if we change the orbit map $f$ from Definition \ref{defn:halfspaceT}, and the halfspace $\h_2$ remains the same up to quasi-isometry if we change the graph $X$ from Definition \ref{defn:treeofspaces}.	
\end{proof}

\begin{lemma}
	The first definition of halfspace is coarsely well-defined and coarsely equivalent to the second definition in the following sense.
	\begin{enumerate}
		\item\label{item:amalThspace} Let $G=A*_C B$ be an amalgam (resp. $G=A*_C$ an HNN extension) as in Definition \ref{defn:halfspaceamal}.
		Then the halfspaces $\h_A,\h_B$ (resp. $\h^+,\h^-$) from Definition \ref{defn:halfspaceamal} remain the same up to quasi-isometry when we change the choice of finite generating sets $S_A,S_B$ (resp. $S_A$).
		Furthermore, if $G\acts T$ is the Bass-Serre tree associated to the given amalgam or HNN decomposition of $G$, then each halfspace from Definition \ref{defn:halfspaceamal} is quasi-isometric to a halfspace of $G\acts T$ from Definition \ref{defn:halfspaceT}, and vice versa.
		\item Let $G\acts T$ be a splitting of a finitely generated group $G$ with minimal action and finitely generated edge stabilizers.
		Fix an edge $e\in ET$ and let $\h\in\cH(T)$ be one of the halfspaces associated to $e$.
		Then $G$ splits as either an amalgam $G=A*_C B$ or an HNN extension $G=A*_C$, with $C=G_e$ in either case, and so that the halfspace from Definition \ref{defn:halfspaceT} corresponding to $\h$ is quasi-isometric to one of the halfspaces $\h_A,\h_B,\h^+,\h^-$ from Definition \ref{defn:halfspaceamal}.
	\end{enumerate}	
\end{lemma}
\begin{proof}
	\begin{enumerate}
		\item We give the proof for the case where $G=A*_C B$ is an amalgam (the HNN case is similar).
		Let $G\acts T$ be the corresponding Bass--Serre tree and let $v_A,v_B\in VT$ be the vertices stabilized by $A$ and $B$ respectively. Let $e$ be the edge joining $v_A$ and $v_B$ and let $\h\subseteq T$ (resp. $\h^*$) be the halfspace bounded by $\hat{e}$ that contains $v_A$ (resp. $v_B$).
		Let $f:G\to VT$ be the orbit map given by $g\mapsto gv_A$ and let $S=S_A\cup S_B$ be the finite generating set for $G$ from Definition \ref{defn:halfspaceamal}.
		It is not hard to show that the halfspace $\h_A$ from Definition \ref{defn:halfspaceamal} is precisely the induced subgraph of $f^{-1}(\h)$ in $\Cay(G,S)$.
		Then, by Lemma \ref{lem:preindependent}, $\h_A$ is quasi-isometric to the halfspace of $G\acts T$ corresponding to $\h$ from Definition \ref{defn:treeofspaces}.
		And then, by Lemma \ref{lem:independent}, $\h_A$ is also quasi-isometric to the halfspace of $G\acts T$ corresponding to $\h$ from Definition \ref{defn:halfspaceT} (with respect to any choice of orbit map and any suitable finite generating set $S'$ for $G$).
		An analogous argument shows that the halfspace $\h_B$ from Definition \ref{defn:halfspaceamal} is quasi-isometric to the halfspace from Definition \ref{defn:halfspaceT} corresponding to $\h^*$.
		Of course the $G$-translate of a given halfspace (with either definition) is isometric to the given halfspace, hence every halfspace from Definition \ref{defn:halfspaceamal} is quasi-isometric to a halfspace from Definition \ref{defn:halfspaceT}, and vice versa.
		Finally, the halfspaces $\h_A,\h_B$ from Definition \ref{defn:halfspaceamal} remain the same up to quasi-isometry when we change the choice of finite generating sets $S_A,S_B$ because the halfspaces from Definition \ref{defn:halfspaceT} are coarsely well-defined (Lemma \ref{lem:independent}).
		
		\item Let $T'$ be obtained from $T$ by collapsing all edges which are not in the orbit of $e$.
		Then $G\acts T'$ is a splitting with a single orbit of edges, so it corresponds to either an amalgam $G=A*_C B$ or an HNN extension $G=A*_C$, with $C=G_e$ in either case.
		Let $e'\in ET$ and $\h'\in\cH(T')$ be the images of $e$ and $\h$ in $T'$.
		If we choose orbit maps $f:G\to VT$ and $f':G\to VT'$ so that $f'$ is the composition of $f$ and the collapse map $T\to T'$, then $f^{-1}(\h)=(f')^{-1}(\h')$, so the halfspace for $G\acts T$ from Definition \ref{defn:halfspaceT} corresponding to $\h$ is identical to the halfspace for $G\acts T'$ from Definition \ref{defn:halfspaceT} corresponding to $\h'$ (assuming we work in the same Cayley graph of $G$ in both cases).
		Then by part \ref{item:amalThspace} of the lemma, these halfspaces are quasi-isometric to either $\h_A,\h_B,\h^+$ or $\h^-$ from Definition \ref{defn:halfspaceamal}.\qedhere		
	\end{enumerate}
\end{proof}

\begin{example}\label{ex:notqiemb}
	Here is an example to show that halfspaces of group splittings are not necessarily quasi-isometrically embedded.
	Let $S$ be a closed surface with negative Euler characteristic and let $\phi:\pi_1(S)\to\pi_1(S)$ be an automorphism induced by a pseudo-Anosov homeomorphism of $S$.
	Let $G=\pi_1(S)\rtimes_\phi\Z$ be the semidirect product induced by $\phi$, and let $p:G\to\Z$ be the projection to the $\Z$ factor.
	Note that $G$ is a hyperbolic group \cite{Thurston22} (indeed it is the fundamental group of a closed hyperbolic 3-manifold).
	The map $p$ induces a splitting of $G$ (an action on $\R$), and the halfspaces correspond to the preimages $p^{-1}((-\infty,n])$ and $p^{-1}([n,\infty))$ for $n\in\Z$ (and their induced subgraphs in the Cayley graph of $G$).
	We claim that at least one of the halfspaces $p^{-1}((-\infty,0])$ and $p^{-1}([0,\infty))$ is not quasi-isometrically embedded in $G$.
	Indeed, suppose they are both quasi-isometrically embedded.
	Then for any $g_1,g_2\in p^{-1}(0)$, there exist quasi-geodesics $\gamma^-$ and $\gamma^+$ (with respect to the metric on $G$) that lie in $p^{-1}((-\infty,0])$ and $p^{-1}([0,\infty))$ respectively and that connect $g_1$ to $g_2$.
	Moreover, the QI-constants of $\gamma^-$ and $\gamma^+$ are independent of $g_1,g_2$.
	By the stability of quasi-geodesics in hyperbolic metric spaces, the quasi-geodesics $\gamma^-$ and $\gamma^+$ are uniformly close to each other, hence uniformly close to $p^{-1}(0)$.
	This implies that $p^{-1}(0)$ is a quasi-convex subgroup of $G$.
	But this is a contradiction, since an infinite normal subgroup of a hyperbolic group is quasiconvex if and only if it has finite index \cite{MihalikTowle94}.
\end{example}

We finish this section with a lemma about the connection between one-endedness of vertex stabilizers and one-endedness of halfspaces.

\begin{lemma}\label{lem:oneendedvertexgroup}
	Let $G\acts T$ be a splitting of a one-ended group $G$ with minimal action and finitely generated edge stabilizers.
	Let $\h\subseteq T$ be a halfspace bounded by an edge midpoint $\hat{e}$, and let $v\in VT$ be the unique vertex in $\h$ incident to $e$.
	If $G_v$ is one-ended then the halfspace of $G\acts T$ corresponding to $\h$ is one-ended.
\end{lemma}
\begin{proof}
	As in Definition \ref{defn:treeofspaces}, let $p:X\to T$ be a tree of spaces for $G\acts T$.
	We must show that $p^{-1}(\h)$ is one-ended, so suppose for contradiction that $K\subseteq p^{-1}(\h)$ is a finite subgraph such that $p^{-1}(\h)-K$ has multiple unbounded components.
	We are assuming that $G$ and $G_v$ are one-ended, hence $X$ and $p^{-1}(v)$ are one-ended as well. In particular $p^{-1}(v)-K$ only has one unbounded component, call it $C_1$.
	Now consider the finite subgraph 
	$$K':=K\cup(p^{-1}(v)-C_1)\subseteq X.$$
	Let $C_2$ be an unbounded component of $p^{-1}(\h)-K$ disjoint from $C_1$.
	Observe that $p^{-1}(\hat{e})$ is contained in a bounded neighborhood of $p^{-1}(v)$ with respect to the induced metric of $p^{-1}(\h)$, hence also in a bounded neighborhood of $C_1$.
	It follows that $C_2\cap p^{-1}(\hat{e})$ is bounded, so only finitely many vertices in $C_2$ lie in $p^{-1}(\hat{e}\cup v)$.
	Any vertex in $C_2-p^{-1}(\hat{e}\cup v)$ is separated from $C_1$ in $X$ by $K'$, so this contradicts one-endedness of $X$.
\end{proof}

\section{Pocsets}\label{sec:pocsets}

We now recall the construction of a CAT(0) cube complex from a pocset or wallspace. This construction was originally due to Sageev \cite{Sageev95}, although our formulation will be closer to the version in \cite{Manning20} -- see also \cite{Nica04,ChatterjiNiblo05,Roller16,WiseRiches}.

\begin{definition}\label{defn:pocset}
	A \emph{pocset} is a poset $(\cP,\leq)$ together with an involution $A\mapsto A^*$ for all $A\in\cP$ satisfying:
	\begin{enumerate}
		\item\label{item:AA*} $A$ and $A^*$ are incomparable.
		\item $A\leq B \Rightarrow B^*\leq A^*$.
	\end{enumerate}
	Elements $A,B\in\cP$ are \emph{nested} if $A\leq B$ or $B\leq A$.
	And $A,B\in\cP$ are \emph{transverse} if neither $A, B$ nor $A, B^*$ are nested. The \emph{width} of $\cP$ is the maximum number of pairwise transverse elements, if such a maximum exists, otherwise we say the width is $\infty$.
	We say that $\cP$ is \emph{discrete} if for any $A\leq B\in\cP$ there are only finitely many $C$ with $A\leq C\leq B$.
\end{definition}

\begin{definition}\label{defn:ultrafilter}
	An \emph{ultrafilter} on a pocset $\cP$ is a subset $\omega\subseteq \cP$ satisfying:
	\begin{enumerate}
		\item (Completeness) For every $A\in \cP$, exactly one of $\{A,A^*\}$ is in $\omega$.
		\item (Consistency) If $A\in\omega$ and $A\leq B$, then $B\in\omega$.
	\end{enumerate}
	An ultrafilter $\omega$ is \emph{DCC} (descending chain condition) if every infinite descending chain $A_1\geq A_2\geq A_3\geq\cdots$ terminates (i.e. there is $N$ with $A_i=A_N$ for all $i\geq N$).
\end{definition}

\begin{proposition}\cite[Definition 9.10, Lemma 9.11 and Theorem 9.16]{Manning20}\label{prop:cubing}
	Let $\cP$ be a pocset of finite width that admits at least one DCC ultrafilter. Then there is a CAT(0) cube complex $C=C(\cP)$, called the \textbf{cubing of $\cP$}, such that:
	\begin{enumerate}
		\item The vertices of $C$ are the DCC ultrafilters on $\cP$.
		\item Two vertices $\omega_1,\omega_2$ in $C$ are joined by an edge if and only if $\omega_1\triangle\omega_2=\{A,A^*\}$ for some $A\in \cP$.
		\item The dimension of $C$ is equal to the width of $\cP$.
	\end{enumerate} 
\end{proposition}

\begin{proposition}\cite[Theorem 2.1]{Dunwoody79}\cite[Theorem 9.1 and Proposition 9.3]{Roller16}\label{prop:cubingtree}
Let $\cP$ be a discrete pocset that contains no pair of transverse elements.
Then $\cP$ admits a DCC ultrafilter.
Hence, by Proposition \ref{prop:cubing}, $C(\cP)$ exists and is a tree.
\end{proposition}
\begin{proof}
	The result follows from the references given above, or it can easily be proved directly as follows.
	Choose $A\in\cP$. Let $\omega=\{B\in\cP\mid A\leq B\}\cup\{B\in\cP\mid A^*\leq B\}$. Since $\cP$ contains no pair of transverse elements, $\omega$ is an ultrafilter on $\cP$. Furthermore, $\omega$ is DCC by the discreteness of $\cP$.
\end{proof}

In this paper we will only deal with pocsets that satisfy Proposition \ref{prop:cubingtree}.
One source of pocsets will come from the following definition.

\begin{definition}\label{defn:wallspace}
	A \emph{wallspace} $(X,\cP)$ is a set $X$ together with a family $\cP$ of non-empty subsets that is closed under complementation, such that for any $x,y\in X$ the set $\{A\in \cP\mid x\in A,\,y\notin A\}$ is finite. $\cP$ forms a pocset under inclusion, with the involution given by complementation. Moreover, for any $x\in X$ the set
	$$\lambda(x):=\{A\in\cP\mid x\in A\}$$
	is a DCC ultrafilter. Therefore, if $\cP$ has finite width, we can form the cubing $C=C(\cP)$, and we have a map $\lambda:X\to C$.
\end{definition}

Another pocset that will frequently be used in this paper is the pocset of halfspaces of a tree, or a tree of spaces. These pocsets are not wallspaces according to Definition \ref{defn:wallspace}, although they are very similar.

\begin{definition}
Let $T$ be a tree. Recall the definition of halfspaces in $T$ (Definition \ref{defn:treehalfspaces}), and let 
$\cH(T)$ denote the set of halfspaces. Let $\h\mapsto \h^*$ be the involution of $\cH(T)$ that exchanges the complementary components of each edge midpoint $\hat{e}$. This makes $\cH(T)$ into a pocset, with partial order given by inclusion, and the cubing of this pocset recovers the tree $T$.
\end{definition}

\begin{definition}
Let $p:X\to T$ be a tree of spaces associated to a group splitting $G\acts T$ (Definition \ref{defn:treeofspaces}).
Let $\cH(X)$ denote the set of halfspaces of $X$ (recall this is just the set of $p$-preimages of halfspaces of $T$), and define an involution $\h\mapsto\h^*$ on $\cH(X)$ so that $p(\h^*)=p(\h)^*$. The partial order of inclusion makes $\cH(X)$ into a pocset, and $p$ induces a pocset isomorphism between $\cH(X)$ and $\cH(T)$.
For $\h\in\cH(X)$, let $\hh:=\h\cap\h^*$, and call this a \emph{wall}. Note that walls are precisely the $p$-preimages of edge midpoints in $T$. Note that halfspaces and walls in $X$ are connected subgraphs of $X$, and the walls are pairwise disjoint.
\end{definition}

Finally, we give the following lemma, to be used in Section \ref{sec:chopping}.

\begin{lemma}\label{lem:hdeep}
	Let $p:X\to T$ be a tree of spaces associated to a group splitting $G\acts T$.
	For each $\h\in\cH(X)$, we have that $\h$ is not contained in a bounded neighborhood of $\hh$ or $\h^*$.
	Moreover, if $G$ (or equivalently $X$) is one-ended, then $\hh$ is unbounded.
\end{lemma}
\begin{proof}
	The action $G\acts T$ is minimal, so $p(\h)$ is not contained in a bounded neighborhood of $p(\hh)$ or $p(\h^*)$.
	Since $p$ is distance non-increasing, we have that $\h$ is not contained in a bounded neighborhood of $\hh$ or $\h^*$.
	Secondly, as $\hh$ separates $\h-\hh$ from $\h^*-\hh$ in $X$, and since $X$ is one-ended, we must have that $\hh$ is unbounded.
\end{proof}

\section{Chopping up halfspaces}\label{sec:chopping}

In this section we prove Theorems \ref{thm:oneendedhs} and \ref{thm:JSJ}.
The main technical result we need is the following.

\begin{theorem}\label{thm:smallerstabs}
	Let $G\acts T$ be a splitting with minimal action and finitely generated edge stabilizers, and with $G$ one-ended and finitely generated.
	If the splitting has some halfspace with more than one end, then there is another splitting $G\acts T'$ with minimal action and finitely generated edge stabilizers, such that every edge stabilizer (resp. vertex stabilizer) of $T'$ is contained in an edge stabilizer (resp. vertex stabilizer) of $T$, and for each $e\in ET$ there is a tree $T_e$ and there is an action of $G$ on $\sqcup_{e\in ET} T_e$ that is compatible with the action on $T$ such that:
	\begin{enumerate}[label={(\alph*)}]
		\item\label{item:Te} For each $e\in ET$, either $T_e$ is a single vertex or the action $G_e\acts T_e$ is a non-trivial finite splitting with finite edge stabilizers.
		\item\label{item:nofix} There is an edge $e_0\in ET$ such that $T_{e_0}$ is not a single vertex.
		\item\label{item:T'toTe} There is an injective $G$-equivariant map $ET'\to\sqcup_e VT_e$.
		\item\label{item:nofixT'} If no vertex stabilizer of $T$ fixes an edge in $T$, then there is a vertex $v_0\in VT$ such that $G_{v_0}$ has no fixed point in $T'$.
	\end{enumerate}
\end{theorem}

\begin{remark}
	One can transform the splitting $G\acts T'$ back to $G\acts T$ by folding (in the sense of \cite{BestvinaFeighn91}). More precisely, there is a vertex $u_0'\in VT'$ and a family of edges $\mathcal{E}'_0$ incident at $u_0'$ such that folding together the edges in each family $g\mathcal{E}'_0$ ($g\in G$) defines a $G$-equivariant map $T'\to T$. Moreover, the image of $\mathcal{E}'_0$ is the edge $e_0\in ET$ from \ref{item:nofix}, and the vertex $v_0\in VT$ from \ref{item:nofixT'} is the endpoint of $e_0$ which is not the image of $u'_0$. 
	These facts can be deduced from the construction in Subsection \ref{subsec:T'}.
\end{remark}

Before proving Theorem \ref{thm:smallerstabs}, let's see how to deduce Theorems \ref{thm:oneendedhs} and \ref{thm:JSJ}.
This requires the following theorem concerning accessible groups.

\begin{theorem}\label{thm:terminate}\cite[Theorem 5.12]{Dicks80}\\
	Let $G$ be a finitely generated accessible group and let $(G_i)$ be a sequence of groups such that $G_0=G$, and $G_{i+1}$ is a vertex group in some non-trivial finite splitting of $G_i$ over finite subgroups. Then the sequence $(G_i)$ terminates.
\end{theorem}

\theoremstyle{plain}
\newtheorem*{thm:oneendedhs}{Theorem \ref{thm:oneendedhs}}
\begin{thm:oneendedhs}
	Let $G\acts T$ be a non-trivial splitting with $G$ one-ended and finitely generated. Suppose the edge stabilizers are finitely generated and accessible.
	Then there is a non-trivial splitting $G\acts T'$ with minimal action such that:
	\begin{enumerate}
		\item The halfspaces of $G\acts T'$ are one-ended.
		\item Edge stabilizers (resp. vertex stabilizers) of $T'$ are finitely generated and are subgroups of the edge stabilizers (resp. vertex stabilizers) of $T$.
		\item For each edge $e'$ in $T'$ there exists an edge $e$ in $T$ such that $G_{e'}$ is a vertex stabilizer in some finite splitting of $G_e$ over finite subgroups.
	\end{enumerate}	
\end{thm:oneendedhs}

\begin{proof}
	Passing to a minimal subtree if necessary, we may assume that the action $G\acts T$ is minimal.	
	If the halfspaces of the splitting are already one-ended then we are done. Otherwise, apply Theorem \ref{thm:smallerstabs} repeatedly until we get a splitting where all the halfspaces are one-ended.
	Each time we apply Theorem \ref{thm:smallerstabs}, the new edge stabilizers (resp. vertex stabilizers) will be finitely generated (using Lemma \ref{lem:fgvertex}), and they will be subgroups of the old edge stabilizers (resp. vertex stabilizers).
	Moreover, by \ref{item:Te} each old edge stabilizer admits a finite splitting over finite subgroups which is either an action on a single vertex or a non-trivial splitting, and at least one of these splittings is non-trivial by \ref{item:nofix}. Furthermore, the new edge stabilizers are vertex stabilizers in the splittings of the old edge stabilizers by \ref{item:T'toTe}.
	Since the original edge stabilizers are accessible, Theorem \ref{thm:terminate} ensures that the process of repeatedly applying Theorem \ref{thm:smallerstabs} cannot go on forever, so we must eventually obtain a splitting of $G$ where all the halfspaces are one-ended.
	The second and third parts of Theorem \ref{thm:oneendedhs} also follow readily from the construction.
\end{proof}

\begin{definition}\label{defn:JSJ} \cite{GuirardelLevitt17}
Let $G$ be a group and let $\cA$ be a family of subgroups of $G$ that is closed under conjugating and taking subgroups.
If $G\acts T$ is a splitting with minimal action and edge stabilizers in $\cA$, then we refer to $T$ as an \emph{$\cA$-tree of $G$}.
An $\cA$-tree $T'$ \emph{dominates} an $\cA$-tree $T$ if every vertex stabilizer of $T'$ fixes a point in $T$.
A \emph{JSJ tree of $G$ over $\cA$} is an $\cA$-tree $T$ such that:
\begin{itemize}
	\item $T$ is \emph{universally elliptic}, meaning its edge stabilizers fix points in all other $\cA$-trees, and
	\item $T$ dominates all other universally elliptic $\cA$-trees.
\end{itemize}
\end{definition}

\theoremstyle{plain}
\newtheorem*{thm:JSJ}{Theorem \ref{thm:JSJ}}
\begin{thm:JSJ}
	Let $G$ be a finitely generated one-ended group and let $\cA$ be a family of subgroups of $G$ that is closed under conjugating and taking subgroups.
\begin{enumerate}
	\item\label{item:existsJSJ} If the groups in $\cA$ are finitely generated and accessible and there exists a JSJ tree $T$ of $G$ over $\cA$, then there exists a JSJ tree $T'$ of $G$ over $\cA$ such that the halfspaces of $G\acts T'$ are one-ended.
	\item\label{item:allJSJ} If $T$ is a JSJ tree of $G$ over $\cA$ with finitely generated edge stabilizers, such that no vertex stabilizer of $T$ fixes an edge in $T$, then the halfspaces of $G\acts T$ are one-ended.
\end{enumerate}
\end{thm:JSJ}

\begin{proof}
	We start with a general observation about JSJ trees.
	Suppose $T$ is a JSJ tree of $G$ over $\cA$, suppose $T'$ is another $\cA$-tree of $G$ that dominates $T$, and suppose that the edge stabilizers of $T'$ are subgroups of the edge stabilizers of $T$. It then follows from Definition \ref{defn:JSJ} that
	\begin{enumerate}[label=(\roman*)]
		\item\label{item:T'JSJ} $T'$ is also a JSJ tree for $G$ over $\cA$, and 
		\item\label{item:TdominatesT'} $T$ dominates $T'$.
	\end{enumerate} 
	
	Let's now show \ref{item:existsJSJ}, so suppose the groups in $\cA$ are finitely generated and accessible and let $T$ be a JSJ tree of $G$ over $\cA$.
	By \ref{item:T'JSJ}, we may apply Theorem \ref{thm:oneendedhs} to produce a JSJ tree $T'$ of $G$ over $\cA$ such that the halfspaces of $G\acts T'$ are one-ended.
	
	Let's now prove \ref{item:allJSJ}, so suppose $T$ is a JSJ tree of $G$ over $\cA$ with finitely generated edge stabilizers, such that no vertex stabilizer of $T$ fixes an edge in $T$, and suppose for contradiction that $G\acts T$ does not have one-ended halfspaces.
	By \ref{item:T'JSJ}, we may apply Theorem \ref{thm:smallerstabs} to produce a JSJ tree $T'$ of $G$ over $\cA$. But part \ref{item:nofixT'} in Theorem \ref{thm:smallerstabs} implies that $T$ does not dominate $T'$, contradicting \ref{item:TdominatesT'}.
\end{proof}

\subsection{Strategy for Theorem \ref{thm:smallerstabs}}\label{subsec:strategy}

The rough strategy to prove Theorem \ref{thm:smallerstabs} is as follows. We take a tree of spaces $p:X\to T$ associated to the splitting $G \acts T$ (Definition \ref{defn:treeofspaces}), and we take a halfspace $\h_0\in\cH(X)$ with more than one end. Say $\h_0$ corresponds to an edge $e_0\in ET$.
We let $\delta$ be some finite collection of edges which cuts $\h_0$ into two unbounded components, and we chop up the halfspace $\h_0$ using the $G_{e_0}$-translates of $\delta$. We define a tree $T_{e_0}$ which is somehow dual to this chopping, together with an action $G_{e_0}\acts T_{e_0}$.
And then we do something similar to obtain the splitting $G\acts T'$, but this time we chop up the entire space $X$ using all the $G$-translates of $\delta$ together with the existing walls $\hh$ in $X$.

The difficulty comes in the details of how we define these trees $T_{e_0}$ and $T'$.
If we were only interested in defining $T_{e_0}$, then we could use the work of Dunwoody \cite[Theorem 1.1]{Dunwoody82} to ensure that any two $G_{e_0}$-translates $g_1\delta,g_2\delta$ induce nested partitions of $\h$, i.e. one of the complementary components of $g_1\delta$ is contained in one of the complementary components of $g_2\delta$, and from there we could define a dual tree $T_{e_0}$ whose edges correspond exactly to the $G_{e_0}$-translates of $\delta$ (using Section \ref{sec:pocsets} or \cite[Theorem 4.1]{Dunwoody82}).
However, the situation becomes more complicated when we try to construct $T'$.
We want to chop up $X$ using the $G$-translates of $\delta$ together with the existing walls $\hh$, but it is not clear how to produce a nested system of partitions of $X$ from this chopping.
The solution is to choose the initial edge cut set $\delta$ in a very careful manner so that the $G_{e_0}$-translates of $\delta$ only chop up the halfspace $\h_0$ and do not chop up any deeper halfspace $\h\subsetneq\h_0$ (Lemma \ref{lem:noh'edges}).
We are then able to construct a system of partitions of $X$ which is compatible with our chopping and which is ``almost nested''.
From this we define a pocset $\cP$ whose elements are genuinely nested (Subsection \ref{subsec:pocsetP}), and we obtain the tree $T'$ as the cubing of $\cP$ (using Section \ref{sec:pocsets} again).
This construction relies on the edge cut set $\delta$ satisfying several properties, so it requires significant work to even construct a suitable $\delta$ (Subsections \ref{subsec:mincut} and \ref{subsec:chopping}).
In particular, in Subsection \ref{subsec:mincut} we use a theorem of Kr{\"o}n \cite{Kron10} (which is a refinement of Dunwoody's theorem \cite[Theorem 1.1]{Dunwoody82} mentioned above).

\subsection{Minimal cuts}\label{subsec:mincut}

In this subsection we recall work of Kr{\"o}n \cite{Kron10} about cutting up graphs with more than one end in an automorphism-invariant way. Kr{\"o}n's paper is in turn based on work of Dunwoody \cite{Dunwoody82}.

\begin{definition}\label{defn:cut}
	Let $Y$ be a connected simplicial graph. For $C\subseteq VY$, define the \emph{edge boundary} of $C$, denoted $\delta C$, to be the set of edges in $Y$ with one vertex in $C$ and one vertex in $VY-C$.
	A \emph{ray} in $Y$ is an embedded one-way infinite path in $Y$.
	A \emph{cut} is a subset $C\subseteq VY$ with $\delta C$ finite and such that the induced subgraphs of $C$ and $VY-C$ are both connected and contain rays.
	Note that, if $Y$ is locally finite, then a connected subgraph of $Y$ contains a ray if and only if it is unbounded.	
	A cut $C$ is \emph{minimal} if $|\delta C|$ is minimal.
\end{definition}

\begin{remark}
	Any locally finite simplicial graph with more than one end has a cut.
\end{remark}

\begin{proposition}\label{prop:cutpocset}
	Let $Y$ be a locally finite simplicial graph.
	Given a cut $C$, we can form a wallspace $(VY,\cP(Y,C))$ (as in Definition \ref{defn:wallspace}), where
	$$\cP(Y,C):=\{gC,VY-gC\mid g\in\Aut(Y)\}.$$
\end{proposition}
\begin{proof}
	Clearly, $\cP(Y,C)$ is a family of non-empty subsets of $VY$ that is closed under complementation.
	It remains to show that, for $x,y\in VY$, the set $S=\{A\in\cP(Y,C)\mid x\in A,\, y\notin A\}$ is finite.
	Indeed, fix a path $\gamma$ from $x$ to $y$. Given $A\in S$, we either have $A=gC$ or $A=VY-gC$ for some $g\in\Aut(Y)$. Now, $\delta (gC)$ determines the pair $(gC,VY-gC)$, so it suffices to show that there are only finitely many possibilities for $\delta (gC)$. But the path $\gamma$ contains an edge in $\delta (gC)$, and $\delta (gC)$ is an $\Aut(Y)$-translate of the finite edge set $\delta C$, so the local finiteness of $Y$ implies that there are only finitely many possibilities for $\delta (gC)$.
\end{proof}

The key result of \cite{Kron10} we will use is the following.

\begin{theorem}\cite[Theorem 3.3]{Kron10}\label{thm:Kron}\\
	Let $Y$ be a locally finite simplicial graph with more than one end. Then $Y$ has a minimal cut $C$ such that the pocset $\cP(Y,C)$ defined in Proposition \ref{prop:cutpocset} has no pairs of transverse elements, or, equivalently, such that $C(\cP(Y,C))$ is a tree.
\end{theorem}

\begin{remark}
	Kr{\"o}n's theorem is actually in the context of simple graphs, but one can deduce the same result for graphs that contain multi-edges.
	Indeed, if a graph $Y$ contains multi-edges then one can consider its first barycentric subdivision $\dot{Y}$, which is a simple graph. There is a correspondence between cuts in $Y$ and cuts in $\dot{Y}$, so if the theorem holds for $\dot{Y}$ it must also hold for $Y$.
\end{remark}

\subsection{Halfspace cuts}\label{subsec:chopping}

We now turn to the proof of Theorem \ref{thm:smallerstabs}, which will be spread over Subsections \ref{subsec:chopping}--\ref{subsec:T'}.
Let $G\acts T$ be a splitting as in the theorem, so a splitting with minimal action and finitely generated edge stabilizers, and with $G$ one-ended and finitely generated.
And we assume that the splitting has some halfspace with more than one end.
Let $p:X\to T$ be a tree of spaces associated to $G\acts T$ (Definition \ref{defn:treeofspaces}).

Define a \emph{halfspace cut} to be a pair $(\h,C)$ where $\h\in\cH(X)$ and $C$ is a cut in $\h$.
As in Definition \ref{defn:cut}, let $\delta C$ denote the set of edges in $\h$ with one vertex in $C$ and one vertex in $V\h-C$.
We are assuming that there exists a halfspace $\h\in\cH(X)$ with more than one end, so there exists $C\subseteq V\h$ such that $(\h,C)$ is a halfspace cut.

\begin{lemma}\label{lem:caphhunbdd}
	If $(\h,C)$ is a halfspace cut, then $C\cap V\hh$ and $V\hh-C$ are both unbounded, and $\delta C$ contains at least one edge in $\hh$.
\end{lemma}
\begin{proof}
	Suppose $C\cap V\hh$ is bounded (argument similar for $V\hh-C$). The wall $\hh$ is unbounded by Lemma \ref{lem:hdeep}, so $V\hh-C$ is unbounded. Let $\delta^+$ be the union of $\delta C$ and all edges in $X$ that are incident to vertices in $C\cap V\hh$. 
	Any path in $X$ that starts in $C$ and leaves $\h$ must cross over $\hh$, and any path in $\h$ that starts in $C$ and reaches $V\h-C$ must cross over $\delta C$.
	Hence $\delta^+$ is a finite set of edges that separates $C$ from $V\h-C$ in $X$, contradicting one-endedness of $X$ (note that $C$ and $V\h-C$ are both unbounded).
	
	Since $C\cap\hh$ and $\hh-C$ are both unbounded, so in particular non-empty, there exists a path in $\hh$ from $C\cap\hh$ to $\hh-C$ (recall that $\hh$ is connected).
	As $\delta C$ separates $C$ from $V\h-C$ in $\h$, this path must contain an edge in $\delta C$.
\end{proof}

Given a halfspace cut $(\h,C)$, let $W(C)$ be the number of edges in $\delta C$ that are contained in walls of $X$. Say that $(\h,C)$ is \emph{$W$-minimal} if $W(C)$ is minimal among all halfspace cuts.

\begin{lemma}\label{lem:oneunbounded}
	Let $(\h,C)$ be a $W$-minimal halfspace cut. If $\delta C$ separates a halfspace $\h'\subsetneq\h$ into more than one component, then only one of these components is unbounded.
\end{lemma}
\begin{proof}
	Suppose for contradiction that $\delta C$ separates a halfspace $\h'\subsetneq\h$ into more than one unbounded component. Let $\delta'\subseteq\delta C\cap E\h'$ be such that $\h'-\delta'$ has exactly two unbounded components, $C'$ and $D'$.
	This means that $(\h',C')$ is a halfspace cut, with $\delta C'=\delta'$. Lemma \ref{lem:caphhunbdd} tells us that $\delta C$ contains an edge in $\hh$, and as this edge is not in $\delta C'$ we deduce that $W(C')<W(C)$, contradicting $W$-minimality of $(\h,C)$.
\end{proof}

\begin{lemma}\label{lem:noh'edges}
	Let $(\h,C)$ be a $W$-minimal halfspace cut. Then each halfspace $\h'\subsetneq\h$ satisfies either $V\h'\subseteq C$ or $V\h'\subseteq V\h-C$.
\end{lemma}
\begin{proof}
	Suppose for contradiction that $\h'\subsetneq\h$ is a halfspace such that $C\cap V\h'$ and $V\h'-C$ are both non-empty.
	By Lemma \ref{lem:oneunbounded}, only one of $C\cap V\h'$ and $V\h'-C$ is infinite.
	Say $C\cap V\h'$ is finite.
	Now $C\cap V\h'$ separates the induced subgraph of $C$ into only finitely many components, so at least one of these components is unbounded -- say it has vertex set $D$.
	As both $\h'$ and the induced subgraph of $V\h-C$ are connected, we see that the induced subgraph of $V\h-D$ is connected. Furthermore, the edge boundary of $D$ in $\h$ is finite since it is contained in the union of $\delta C$ and the set of edges that meet $C\cap V\h'$.
	Hence $(\h,D)$ is a halfspace cut.
	
	We now obtain a contradiction by showing that $W(D)<W(C)$.
	Indeed, $\delta D-\delta C$ only contains edges that join a vertex in $D\subseteq \h-\h'$ to a vertex in $C\cap V\hh'$, and these edges are not contained in walls.
	On the other hand, any path in the induced subgraph of $C$ from $\h'$ to $C-\h'$ crosses $\hh'$, so $C\cap V\hh'\neq\emptyset$, and any path in $\hh'$ from $C\cap V\hh'$ to $V\hh'-C$ crosses $\delta C$, so $\delta C-\delta D$ contains at least one edge in the wall $\hh'$.
\end{proof}

\begin{figure}[H]
	\centering
		\begin{tikzpicture}[auto,node distance=2cm,
			thick,every node/.style={},
			every loop/.style={min distance=2cm},
			hull/.style={draw=none},
			scale=.8
			]
			\tikzstyle{label}=[draw=none]

			\draw[rounded corners=20pt] (-2,8)--(-2,6)--(4,6)--(4,3)--(2,3)--(2,4);
			\draw[rounded corners=20pt] (1,4)--(1,2)--(6,2);
			
			\draw[fill=black!30!green,opacity=.2](0,0)--(0,6){[rounded corners=20]--(4,6)--(4,3)--(2,3)}--(2,4)--(1,4){[rounded corners=20]--(1,2)}--(6,2)--(6,0)--(0,0);
			\draw[fill=black,opacity=.2](-6,-1)--(-6,8)--(-2,8){[rounded corners=20]--(-2,6)}--(0,6)--(0,0)--(6,0)--(6,-1)--(-6,-1);

			\draw[red,line width=1pt] (-6,0)--(6,0);
			\draw[draw=red,fill=black,-triangle 90, ultra thick](-3,0)--(-3,1);			
			\node[label,red] at (-2.4,.5){$\h$};
			
			\draw[red,line width=1pt] (-6,4)--(4,4);
			\draw[draw=red,fill=black,-triangle 90, ultra thick](-3,4)--(-3,5);			
			\node[label,red] at (-2.4,4.5){$\h'$};
			
			\node[label,black!30!green] at (4.7,4){$C$};
			\node[label] at (2,1){$D$};
			
		\end{tikzpicture}
	
	\caption{Cartoon picture for the proof of Lemma \ref{lem:noh'edges}.
	Here $D$ is the vertex set of an unbounded component of the induced subgraph of $C-\h'$.}\label{fig:CD}
\end{figure}

\begin{lemma}\label{lem:modX}
	Modifying $X$ if necessary, we may assume that all halfspace cuts $(\h,C)$ with $|\delta C|$ minimal are also $W$-minimal.
\end{lemma}
\begin{proof}
	The modification of $X$ we consider is where we fix an integer $n\geq1$, and replace each edge in each wall of $X$ with $n$ edges that join the same pair of vertices. Let $X_n$ be the resulting graph, which is endowed with a collection of halfspaces and walls that are inherited from $X$ in the obvious way.
	Any halfspace cut $(\h,C)$ in $X$ is still a halfspace cut in $X_n$, and vice versa.
	The only difference is that the edge boundary $\delta C$ might become larger in $X_n$ because of the extra edges.
	More precisely, the numbers $W(C)$ and $|\delta C|$ both increase by $(n-1)W(C)$ when passing from $X$ to $X_n$.
	Note that a halfspace cut is $W$-minimal with respect to $X$ if and only if it is $W$-minimal with respect to $X_n$, but the property of having minimal $|\delta C|$ may change when passing between $X$ and $X_n$.
	Let $w$ be the smallest value of $W(C)$ among halfspace cuts $(\h,C)$ in $X$, and let $N$ be the smallest value of $|\delta C|$ among $W$-minimal halfspace cuts $(\h,C)$ in $X$.
	Put $n=N+1$.
	Then the smallest possible value of $|\delta C|$ in $X_n$ is $N+(n-1)w$, and this is only achieved by $W$-minimal halfspace cuts, as required.
\end{proof}

Now let $(\h_0,C)$ be a halfspace cut with $|\delta C|$ minimal.
Lemmas \ref{lem:noh'edges} and \ref{lem:modX} imply that each halfspace $\h\subsetneq\h_0$ satisfies either $V\h\subseteq C$ or $V\h\subseteq V\h_0-C$.
Furthermore, by Theorem \ref{thm:Kron} we can assume that the pocset $\cP(\h_0,C)$ has no pairs of transverse elements.

\subsection{The trees $T_e$}\label{subsec:Te}

For brevity we will write $G_0=G_{\h_0}$ for the $G$-stabilizer of $\h_0$ and we will write $\delta$ for the edge boundary $\delta C$.
Note that $\h_0$ corresponds to some edge $e_0\in ET$ (i.e. $p(\h_0)\subseteq T$ is bounded by $\hat{e}_0$) and $G_0=G_{e_0}$.
As in Subsection \ref{subsec:mincut}, we can form the pocset $\cP(\h_0,C)$.
Now consider the subpocset
$$\cP_0:=\{gC,V\h_0-gC\mid g\in G_0\}\subseteq\cP(\h_0,C).$$ 
Recall that the partial order on $\cP_0$ is inclusion and the involution is complementation in $V\h_0$, i.e. the map $D\mapsto V\h_0-D$.
We know that $\cP(\h_0,C)$ has no pairs of transverse elements, so $\cP_0$ also has no pairs of transverse elements.
Therefore, the cubing of $\cP_0$ is a tree, call it $T_0$.

\begin{lemma}\label{lem:T0}
	$G_0$ acts on $T_0$ cocompactly and with finite edge stabilizers.
\end{lemma}
\begin{proof}
	The edges of $T_0$ correspond to the pairs $\{gC,V\h_0-gC\}$ for $g\in G_0$, or equivalently to the edge boundaries $g\delta$ for $g\in G_0$. Hence $G_0$ acts on $T_0$ with a single orbit of edges and at most two orbits of vertices. Since $G_0$ acts properly on $\h_0$, and since the sets of edges $g\delta$ are finite, we deduce that $G_0$ acts on $T_0$ with finite edge stabilizers.
\end{proof}

\begin{lemma}\label{lem:T0nofixed}
	The action $G_0\acts T_0$ has no fixed point.
\end{lemma}
\begin{proof}	
	It suffices to find $D\in\cP_0$ and $g\in G_0$ with $gD\subsetneq D$, as then $g$ will have no fixed point in $T_0$. Indeed, $g$ will translate along the axis of $T_0$ that contains the edges corresponding to $\{g^iD,V\h_0-g^iD\}$ for $i\in\Z$. 
	
	Recall that $C\cap V\hh_0$ and $V\hh_0-C$ are both unbounded by Lemma \ref{lem:caphhunbdd}. As $G_0$ acts cocompactly on $\hh_0$, and as the $G_0$-translates of $\delta$ all contain edges in $\hh_0$ (Lemma \ref{lem:caphhunbdd} again), there exist $g_1,g_2\in G_0$ with $g_1\delta$ contained in the induced subgraph of $C$ and $g_2\delta$ contained in the induced subgraph of $V\h_0-C$. 
	We deduce that one of the components of $\h_0-g_1 \delta$ contains both $\delta$ and $V\h_0-C$, and the other component is contained in the induced subgraph of $C$.
	So either $g_1C\subsetneq C$ or $V\h_0-C\subsetneq g_1C$.
	Similarly, we have $g_2(V\h_0-C)\subsetneq V\h_0-C$ or $C\subsetneq g_2(V\h_0-C)$.
	If $g_1C\subsetneq C$ or $g_2(V\h_0-C)\subsetneq V\h_0-C$ then we are done, otherwise $V\h_0-C\subsetneq g_1C$ and $C\subsetneq g_2(V\h_0-C)$, so $g_2C\subsetneq V\h_0-C\subsetneq g_1 C$, and we are again done because $g_1^{-1}g_2C\subsetneq C$.
\end{proof}

If $\{g_i\mid i\in\Omega\}$ is a left transversal of $G_0$ in $G$ then we get an induced action of $G$ on the product
$T_0\times G/G_0$, explicitly this is given by
\begin{equation}\label{inducedaction}
	g\cdot(v,g_iG_0):=(g_0v,g_jG_0),
\end{equation}
where $gg_i=g_jg_0$ with $i,j\in\Omega$ and $g_0\in G_0$, and $g_0v$ refers to the action of $G_0$ on $T_0$. 
(This construction is essentially the same as the notion of induced representation from representation theory.)
We may assume that the transversal $\{g_i\}$ includes the identity element, in which case the action of $G_0$ on $T_0\times\{G_0\}$ recovers the original action of $G_0$ on $T_0$.

We can then define the trees $T_e$ required for Theorem \ref{thm:smallerstabs}, and the action of $G$ on $\sqcup_eT_e$. Indeed, recall that there is an edge $e_0\in ET$ corresponding to $\h_0$ (i.e. $p(\h_0)\subseteq T$ is bounded by $\hat{e}_0$). Put 
$$T_{g_ie_0}:=T_0\times\{g_i G_0\}$$
for each $i\in\Omega$.
Note that the $G$-stabilizer of $T_{e_0}$ is $G_0=G_{e_0}$, and the action $G_0\acts T_{e_0}$ is a non-trivial finite splitting with finite edge stabilizers by Lemmas \ref{lem:T0} and \ref{lem:T0nofixed}.
Similarly, the $G$-stabilizer of $T_{g_ie_0}$ is $g_iG_0g_i^{-1}=G_{g_ie_0}$, and the action $G_{g_ie_0}\acts T_{g_ie_0}$ is also a non-trivial finite splitting with finite edge stabilizers because it is conjugate to the action $G_0\acts T_{e_0}$.
For $e\notin G\cdot e_0$ we define $T_e$ to be a single vertex and we define $gT_e=T_{ge}$.
Altogether this gives us an action of $G$ on $\sqcup_eT_e$ that is compatible with the action on $T$, and that satisfies properties \ref{item:Te} and \ref{item:nofix} from Theorem \ref{thm:smallerstabs}.

\subsection{An equivalence relation on $V\h_0$}

Define an equivalence relation on $V\h_0$ where $x\sim y$ if $x$ and $y$ are not separated in $\h_0$ by any $g\delta$ with $g\in G_0$. Equivalently, $x\sim y$ if there is no $D\in\cP_0$ with $x\in D$ and $y\notin D$.
Let $[x]$ denote the equivalence class of $x$, and let $[x]^*=VX-[x]$.

Since $T_0$ is the cubing of $\cP_0$, and $\cP_0$ forms a wallspace $(V\h_0,\cP_0)$, we have a map $\lambda:V\h_0\to VT_0$ as in Definition \ref{defn:wallspace}. Note that the equivalence classes $[x]$ are just the non-empty fibers of $\lambda$.
We now prove two lemmas about how these equivalence classes interact with halfspaces $\h\subsetneq\h_0$.

\begin{lemma}\label{lem:hin[x]}
	Each halfspace $\h\subsetneq\h_0$ is contained in a $\sim$-class $[x]$.
\end{lemma}
\begin{proof}
	Fix $\h\subsetneq\h_0$. Given $D\in\cP_0$, we can write $D=gC$ or $D=V\h_0-gC$ for some $g\in G_0$.	
	Note that $g^{-1}\h$ is another halfspace strictly contained in $\h_0$, so our choice of $(\h_0,C)$ at the end of Subsection \ref{subsec:chopping}	implies that we either have $g^{-1}V\h\subseteq C$ or $g^{-1}V\h\subseteq V\h_0-C$.
	Left multiplying by $g$ then implies that we either have $V\h\subseteq D$ or $V\h\subseteq V\h_0-D$.
	This holds for all $D\in \cP_0$, so the lemma follows.
\end{proof}

\begin{lemma}\label{lem:[y]*preceq}
	Suppose $g\in G$ with $g\h_0^*\subseteq\h_0$. 
	Then there exist $\sim$-classes $[x]$ and $[y]$ with $gV\h_0^*\subseteq[x]$ and $V\h_0^*\subseteq g[y]$.
	Furthermore, $g[y]^*$ is contained in a bounded neighborhood of $[x]$.
\end{lemma}
\begin{proof}
	Note that $g\h_0^*\subsetneq\h_0$ since $G$ acts on $T$ without edge inversions.
	The existence of $[x]$ with $gV\h_0^*\subseteq[x]$ then follows immediately from Lemma \ref{lem:hin[x]}.
	Meanwhile, $g\h_0^*\subsetneq\h_0$ implies $\h_0^*\subsetneq g\h_0$ by taking complementary halfspaces on both sides.
	Hence $g^{-1}\h_0^*\subsetneq\h_0$, and there exists $[y]$ with $g^{-1}V\h_0^*\subset[y]$ by Lemma \ref{lem:hin[x]}.
	Applying $g$ to both sides yields the desired inclusion $V\h_0^*\subseteq g[y]$.
	
	It remains to prove the second part of the lemma.
	By definition, $[y]^*$ is the union of $V\h_0^*-V\h_0$ and the sets $[z]$ for $\sim$-classes $[z]\neq[y]$.
	Hence $g[y]^*$ is contained in the union of $g\h_0^*$ and $g[z]$ for $\sim$-classes $[z]\neq[y]$. We already know that $gV\h_0^*\subseteq[x]$, so it remains to show that the sets $g[z]$ for $[z]\neq[y]$ are contained in a bounded neighborhood of $[x]$.
	
	Let $z\in V\h_0$ with $[z]\neq[y]$.
	Suppose $gz\notin[x]$.	
	Suppose $g_0\delta$ separates $[z]$ from $[y]$ in $\h_0$, with $g_0\in G_0$.
	So $gg_0\delta$ separates $gz$ from $g[y]$ in $g\h_0$.
	Since $V\h_0^*\subseteq g[y]$, we have that $gz\in\h_0$ and that $gg_0\delta$ separates $gz$ from $\h_0^*$ in $g\h_0$ (see the left-hand picture in Figure \ref{fig:gz}).
	On the other hand, as $gz\notin[x]$, there is some $g_1\in G_0$ such that $g_1\delta$ separates $gz$ from $[x]$ in $\h_0$ (see the right-hand picture in Figure \ref{fig:gz}). 
	And as $gV\h_0^*\subseteq[x]$, we know that $g_1\delta$ also separates $gz$ from $g\h_0^*$ in $\h_0$.
	
	The subgraph of $X$ spanned by vertices that lie in $\h_0-\hh_0$ and not in halfspaces $\h\subsetneq\h_0$ is the $p$-preimage of a vertex in $T$, so it is a connected subgraph (see Definition \ref{defn:treeofspaces}). 
	By construction of $(\h_0,C)$ at the end of Subsection \ref{subsec:chopping}, each edge in $\delta$ is contained in $\h_0$ but not in any of the halfspaces $\h\subsetneq\h_0$ (although it might have one endpoint in $\h$).
	As a result, there exists a finite connected subgraph $F\subseteq\h_0$ that contains $\delta$ and such that $F\subseteq\h^*$ for any halfspace $\h\subsetneq\h_0$.
	Let $L$ be the diameter of $F$.
	Note that $F$ and $L$ do not depend on the choice of $z$.
	
	Let $\gamma$ be a path in $X$ from $gz$ to $[x]$.
	We claim that $\gamma$ intersects the $L$-neighborhood of $gg_0\delta$.
	\begin{enumerate}
		\item\label{gammacaseA} Suppose $\gamma$ intersects $\h_0^*\cup g\h_0^*$.
			Let $z'$ be the first point of $\gamma$ that lies in $\h_0^*\cup g\h_0^*$.
			So the initial segment $\gamma'\subseteq\gamma$ from $gz$ to $z'$ lies in $\h_0\cap g\h_0$.
			As $gg_0\delta$ separates $gz$ from $\h_0^*$ in $g\h_0$, and $g_1\delta$ separates $gz$ from $g\h_0^*$ in $\h_0$, it must be that $\gamma'$ intersects either $gg_0\delta$ or $g_1\delta$.
			If $\gamma'$ intersects $gg_0\delta$ then the claim is proved, so suppose $\gamma'$ intersects $g_1\delta$ but not $gg_0\delta$.
			
			Observe that $\gamma'\cap g_1\delta\subseteq g_1 F$.
			By Lemma \ref{lem:caphhunbdd}, $g_1\delta$, and hence $g_1 F$, intersects $\hh_0$.
			By definition, $F\subseteq\h^*$ for any halfspace $\h\subsetneq\h_0$, so $g_1F$ has the same property; in particular $g_1F\subseteq g\h_0$.
			As $gg_0\delta$ separates $gz$ from $\h_0^*$ in $g\h_0$, it must also separate $\gamma'\cap g_1\delta$ from $\h_0^*$ in $g\h_0$.
			Since $g_1F$ is connected and lies in $g\h_0$, we deduce that $gg_0\delta$ separates $\gamma'\cap g_1\delta$ from $\h_0^*$ in $g_1F$.
			As $L$ is the diameter of $F$, and hence $g_1F$, we see that $\gamma'\cap g_1\delta$ lies in the $L$-neighborhood of $gg_0\delta$, as required.
			
		\item Now suppose $\gamma$ does not intersect $\h_0^*\cup g\h_0^*$.
		Then $\gamma$ is contained in $\h_0\cap g\h_0$.
		As $g_1\delta$ separates $gz$ from $[x]$ in $\h_0$, it must be that $\gamma$ intersects $g_1\delta$.
		If $\gamma$ also intersects $gg_0\delta$ then the claim is proved, so suppose $\gamma$ intersects $g_1\delta$ but not $gg_0\delta$.
		We can then argue as in the second half of part \ref{gammacaseA} (with $\gamma$ in place of $\gamma'$) to conclude that $\gamma$ intersects the $L$-neighborhood of $gg_0\delta$. This completes the proof of the claim.
	\end{enumerate}
	
	We have shown that either $d(gz,gg_0\delta)\leq L$ or the $L$-neighborhood of $gg_0\delta$ separates $gz$ from $[x]$ in $X$. Suppose for the moment we are in the latter case.
	As $X$ is one-ended, the $L$-neighborhood of $\delta$ only has one unbounded complementary component.
	Since $X$ is locally finite, there are only finitely many bounded complementary components, so there is a constant $K\geq L$ such that the bounded complementary components are contained in the $K$-neighborhood of $\delta$.
	Since $gV\h_0^*\subseteq[x]$, we know that $[x]$ is unbounded, so $gz$ is in a bounded complementary component of the $L$-neighborhood of $gg_0\delta$, and $d(gz,gg_0\delta)\leq K$.
	Thus in either case we get $d(gz,gg_0\delta)\leq K$.
	Now $gg_0 F$ contains $gg_0\delta$ and intersects $g\h_0^*$, and $gg_0F$ has diameter $L$, so $d(gz,g\h_0^*)\leq K+L$. In turn this implies $d(gz,[x])\leq K+L$.
	As $K$ and $L$ do not depend on the choice of $z$, this implies that all the sets $g[z]$ with $z\in V\h_0$ and $[z]\neq[y]$ are contained in a bounded neighborhood of $[x]$, as required.
\end{proof}

\begin{figure}[H]

	\begin{tikzpicture}[auto,node distance=2cm,
		thick,every node/.style={},
		every loop/.style={min distance=2cm},
		hull/.style={draw=none},
		scale=.55
		]
		\tikzstyle{label}=[draw=none]

		\draw[rounded corners=20pt] (-6,4)--(-4,4)--(-4,1)--(-6,1);
		\draw[rounded corners=20pt] (1,6)--(1,1)--(6,1);
		\draw[rounded corners=20pt] (2.5,6)--(2.5,2)--(6,2);
		
		\draw[fill=black!30!green,opacity=.2](-6,-1)--(-6,1){[rounded corners=20]--(-4,1)}--(-4,6)--(1,6){[rounded corners=20]--(1,1)}--(6,1)--(6,-1)--(-6,-1);
		\draw[fill=black,opacity=.2](-6,6)--(-4,6)--(-4,2.5){[rounded corners=20]--(-4,4)}--(-6,4)--(-6,6);
		\draw[fill=black,opacity=.2](2.5,6){[rounded corners=20]--(2.5,2)}--(6,2)--(6,6)--(2.5,6);
		\draw[fill=black,opacity=.2](-6,6)--(6,6)--(6,7)--(-6,7)--(-6,6);
		
		\draw[blue,line width=1pt] (1,6) rectangle (2.5,4.5);
		\draw[fill=blue,opacity=.5] (1,6) rectangle (2.5,4.5);
		\node[label,blue] at (3.3,5.1){$gg_0\delta$};
		

		\draw[red,line width=1pt] (-6,0)--(6,0);
		\draw[draw=red,fill=black,-triangle 90, ultra thick](-3,0)--(-3,1);			
		\node[label,red] at (-2.2,.8){$\h_0$};
		
		\draw[red,line width=1pt] (-6,6)--(6,6);
		\draw[draw=red,fill=black,-triangle 90, ultra thick](-3,6)--(-3,5);			
		\node[label,red] at (-2.1,5){$g\h_0$};
		
		\node[circle,draw,fill,inner sep=0pt,minimum size=5pt] at (4,3){};
		
		\node[label,black!30!green] at (-1,3){$g[y]$};
		\node[label] at (4.7,3){$gz$};
		
	\end{tikzpicture}
\hspace{1cm}
	\begin{tikzpicture}[auto,node distance=2cm,
	thick,every node/.style={},
	every loop/.style={min distance=2cm},
	hull/.style={draw=none},
	scale=.55
	]
	\tikzstyle{label}=[draw=none]

	\draw[rounded corners=20pt] (-6,5)--(-4,5)--(-4,2)--(-6,2);
	\draw[rounded corners=20pt] (1,0)--(1,5)--(6,5);
	\draw[rounded corners=20pt] (2.5,0)--(2.5,4)--(6,4);

	\draw[fill=black!30!green,opacity=.2](-6,7)--(-6,5){[rounded corners=20]--(-4,5)}--(-4,0)--(1,0){[rounded corners=20]--(1,5)}--(6,5)--(6,7)--(-6,7);
	\draw[fill=black,opacity=.2](-6,0)--(-4,0)--(-4,3.5){[rounded corners=20]--(-4,2)}--(-6,2)--(-6,0);
	\draw[fill=black,opacity=.2](2.5,0){[rounded corners=20]--(2.5,4)}--(6,4)--(6,0)--(2.5,0);
	\draw[fill=black,opacity=.2](-6,0)--(6,0)--(6,-1)--(-6,-1)--(-6,0);
	
	\draw[blue,line width=1pt] (1,0) rectangle (2.5,1.5);
	\draw[fill=blue,opacity=.5] (1,0) rectangle (2.5,1.5);
	\node[label,blue] at (1.7,2){$g_1\delta$};
	

	\draw[red,line width=1pt] (-6,0)--(6,0);
	\draw[draw=red,fill=black,-triangle 90, ultra thick](-3,0)--(-3,1);			
	\node[label,red] at (-2.2,.8){$\h_0$};
	
	\draw[red,line width=1pt] (-6,6)--(6,6);
	\draw[draw=red,fill=black,-triangle 90, ultra thick](-3,6)--(-3,5);			
	\node[label,red] at (-2.1,5){$g\h_0$};
	
	\tikzstyle{a}=[isosceles triangle,fill,sloped,allow upside down,minimum size=5pt,inner sep=0pt,rotate=180,draw]
	
	\draw[rounded corners=20pt] (4,3)--(4,1)--(-1,1);
	\node[a] at (1.5,1){};
	\node[circle,draw,fill,inner sep=0pt,minimum size=5pt] at (-1,1){};
	
	\node[circle,draw,fill,inner sep=0pt,minimum size=5pt] at (4,3){};
	
	\node[label,black!30!green] at (-1,3){$[x]$};
	\node[label] at (4.7,3){$gz$};
	\node[label] at (4.3,1.5){$\gamma$};
	
\end{tikzpicture}

	\caption{Cartoon picture for the proof of Lemma \ref{lem:[y]*preceq}.
	The two pictures appear mutually contradictory, and indeed the proof shows that this situation can only happen when $gz$ is uniformly close to both $gg_0\delta$ and $g_1\delta$, so in reality the picture would look more like Figure \ref{fig:gzreality}.}\label{fig:gz}
\end{figure}

\begin{figure}[H]

	\begin{tikzpicture}[auto,node distance=2cm,
		thick,every node/.style={},
		every loop/.style={min distance=2cm},
		hull/.style={draw=none},
		scale=.55
		]
		\tikzstyle{label}=[draw=none]

		\draw[rounded corners=20pt] (-6,5)--(-4,5)--(-4,2)--(-6,2);
		\draw[rounded corners=20pt] (1,3)--(1,5)--(6,5);
		\draw[rounded corners=20pt] (1.5,3)--(3,3)--(3,2);
		\draw[rounded corners=20pt] (3,1)--(6,1);

		\draw[fill=black!30!green,opacity=.2](-6,7)--(-6,5){[rounded corners=20]--(-4,5)}--(-4,0)--(-2,0)--(-2,3)--(-1,3)--(-1,6)--(0,6)--(0,3)--(1,3){[rounded corners=20]--(1,5)}--(6,5)--(6,7)--(-6,7);
		\draw[fill=black,opacity=.2](-6,0)--(-4,0)--(-4,3.5){[rounded corners=20]--(-4,2)}--(-6,2)--(-6,0);
		\draw[fill=black,opacity=.2](1.5,3){[rounded corners=20]--(3,3)}--(3,2)--(1.5,2)--(1.5,3);
		\draw[fill=black,opacity=.2](-6,0)--(1.5,0)--(1.5,1)--(6,1)--(6,-1)--(-6,-1)--(-6,0);
		
		\draw[blue,line width=1pt] (-2,0) rectangle (1.5,3);
		\draw[fill=blue,opacity=.5] (-2,0) rectangle (1.5,3);
		\node[label,blue] at (-2.6,2){$g_1\delta$};
		
		\draw[blue,line width=1pt](-1,6)--(0,6)--(0,2)--(3,2)--(3,1)--(-1,1)--(-1,6);
		\draw[fill=blue,opacity=.5](-1,6)--(0,6)--(0,2)--(3,2)--(3,1)--(-1,1)--(-1,6);
		\node[label,blue] at (.8,5.1){$gg_0\delta$};
		

		\draw[red,line width=1pt] (-6,0)--(6,0);
		\draw[draw=red,fill=black,-triangle 90, ultra thick](-3,0)--(-3,1);			
		\node[label,red] at (-3.6,1){$\h_0$};
		
		\draw[red,line width=1pt] (-6,6)--(6,6);
		\draw[draw=red,fill=black,-triangle 90, ultra thick](-3,6)--(-3,5);			
		\node[label,red] at (-2.1,5){$g\h_0$};
		
		\node[circle,draw,fill,inner sep=0pt,minimum size=5pt] at (2.4,2.5){};
		
		\node[label,black!30!green] at (-3,3.5){$[x]$};
		\node[label] at (3,3){$gz$};
		
	\end{tikzpicture}
	
	\caption{Cartoon picture for the proof of Lemma \ref{lem:[y]*preceq}.
		A depiction of how the situation might actually look.}\label{fig:gzreality}
\end{figure}

\subsection{The pocset $\cP$}\label{subsec:pocsetP}

We now build a pocset $\cP$ using the halfspaces of $X$ and the $\sim$-classes.
The partial order on $\cP$ will be based on \emph{coarse inclusion} of subsets of $X$.
Specifically, for subsets $\fa,\fb\subseteq X$, write $\fa\preceq\fb$ if $\fa$ is contained in a bounded neighborhood of $\fb$.
Note that $\fa\subseteq\fb$ implies $\fa\preceq\fb$, and $\fa\preceq\fb\preceq\fc$ implies $\fa\preceq\fc$.
By convention, $\emptyset\preceq\fb$ for any $\fb$.

Let
$$\cH^\#=\{\h\in\cH(X)\mid\hh\notin G\cdot\hh_0\},$$
and let $\cM_0$ denote the set of $\sim$-classes in $V\h_0$ that are not contained in a bounded neighborhood of $\hh_0$.
Note that the $\sim$-class $[x]$ appearing in Lemma \ref{lem:hin[x]} is in $\cM_0$, because the halfspace $V\h\subseteq[x]$ is not contained in a bounded neighborhood of $\hh_0$.
The same goes for the $\sim$-classes $[x]$ and $[y]$ appearing in Lemma \ref{lem:[y]*preceq}. 
We also get the following lemma.

\begin{lemma}\label{lem:[x]npreceq[x]*}
	We cannot have $[x]\preceq[x]^*$ for $[x]\in\cM_0$.
\end{lemma}
\begin{proof}
Any path from $[x]$ to $[x]^*$ must either pass through $\hh_0$ or through a $G_0$-translate of $\delta$, which is contained in a bounded neighborhood of $\hh_0$.
The lemma follows since $[x]$ is not contained in a bounded neighborhood of $\hh_0$ (which is by definition of $\cM_0$).
\end{proof}

Now define $\cP$ to be the set of all pairs $(\fa,\h)$, with $\fa\subseteq VX$ and $\h\in\cH(X)$, that arise in one of the following three ways:
\begin{equation}
	(\fa,\h)=
	\begin{cases}
		(V\h,\h),              & \h\in\cH^\#,\\
		(g[x],g\h_0),&g\in G,\,[x]\in\cM_0,\\
		(g[x]^*,g\h_0^*),&g\in G,\,[x]\in\cM_0,	
	\end{cases}
\end{equation}

Define a relation $\leq$ on $\cP$, where $(\fa_1,\h_1)\leq(\fa_2,\h_2)$ if
\begin{enumerate}
	\item\label{item:prec} $\fa_1\prec\fa_2$ (meaning $\fa_1\preceq\fa_2$ but $\fa_2\npreceq\fa_1$), or
	\item\label{item:preceq} $\fa_1\preceq\fa_2$ and $\h_1\subseteq\h_2$.
\end{enumerate} 
\begin{remark}
	It can happen that conditions \ref{item:prec} and \ref{item:preceq} both hold.
\end{remark}

\begin{lemma}\label{lem:[x]leq[y]}
	For distinct $[x],[y]\in\cM_0$, we have
	$([x],\h_0)\leq([y]^*,\h_0^*)$, but we do not have $([x],\h_0)\leq([y],\h_0)$ or $([x]^*,\h_0^*)\leq([y]^*,\h_0^*)$.
\end{lemma}
\begin{proof}
	We know that $[x]$ and $[y]$ are disjoint, so $[x]\preceq[y]^*$. Moreover, $V\h_0^*-\hh_0\subseteq[y]^*$ is not contained in a bounded neighborhood of $[x]$, so $[x]\prec[y]^*$ and $([x],\h_0)\leq([y]^*,\h_0^*h)^*$ is an inequality of type \ref{item:prec}.
	
	Secondly, we cannot have $([x],\h_0)\leq([y],\h_0)$ as that would imply $[x]\preceq[y]\subseteq[x]^*$, contradicting Lemma \ref{lem:[x]npreceq[x]*}.
	Similarly, we cannot have $([x]^*,\h_0^*)\leq([y]^*,\h_0^*)$ as that would imply $[y]\subseteq[x]^*\preceq[y]^*$, again contradicting Lemma \ref{lem:[x]npreceq[x]*}.
\end{proof}

\begin{lemma}\label{lem:partial}
	$\leq$ is a partial order on $\cP$.
\end{lemma}
\begin{proof}
	Reflexivity holds by \ref{item:preceq}. Let's now prove antisymmetry. If $(\fa_1,\h_1)\leq(\fa_2,\h_2)\leq(\fa_1,\h_1)$ then $\fa_1\preceq\fa_2\preceq\fa_1$, so both $\leq$ inequalities must be of type \ref{item:preceq}, which means that $\h_1\subseteq\h_2\subseteq\h_1$, and so $\h_1=\h_2$.
	If $\h_1\in\cH^\#$ then $(\fa_1,\h_1)=(\fa_2,\h_2)=(V\h_1,\h_1)$.
	If $\h_1=g\h_0$ (some $g\in G$), then $(\fa_1,\h_1)=(g[x],g\h_0)$ and $(\fa_2,\h_2)=(g[y],g\h_0)$ for some $[x],[y]\in\cM_0$, and $[x]=[y]$ by Lemma \ref{lem:[x]leq[y]}.
	Similarly, if $\h_1=g\h_0^*$ (some $g\in G$), then $(\fa_1,\h_1)=(g[x]^*,g\h_0^*)$ and $(\fa_2,\h_2)=(g[y]^*,g\h_0^*)$ for some $[x],[y]\in\cM_0$, and $[x]=[y]$ by Lemma \ref{lem:[x]leq[y]}.
	
	Finally, we prove transitivity. Suppose $(\fa_1,\h_1)\leq(\fa_2,\h_2)\leq(\fa_3,\h_3)$. If one of the inequalities is of type \ref{item:prec}, then $\fa_1\prec\fa_3$ and so $(\fa_1,\h_1)\leq(\fa_3,\h_3)$ is an inequality of type \ref{item:prec}. On the other hand, if both inequalities are of type \ref{item:preceq}, then $(\fa_1,\h_1)\leq(\fa_3,\h_3)$ is also an inequality of type \ref{item:preceq}.
\end{proof}

Define an involution $(\fa,\h)\mapsto(\fa,\h)^*$ on $\cP$ by 
\begin{align*}
	(V\h,\h)&\mapsto(V\h^*,\h^*),              & \h\in\cH^\#,\\
	(g[x],g\h_0)&\mapsto(g[x]^*,g\h_0^*),&g\in G,\,[x]\in\cM_0,\\
	(g[x]^*,g\h_0^*)&\mapsto(g[x],g\h_0),&g\in G,\,[x]\in\cM_0.		
\end{align*}
We have an action of $G$ on $\cP$ given by $g\cdot(\fa,\h):=(g\fa,g\h)$, which clearly respects $\leq$ and $*$.
We will show that $(\cP,\leq,*)$ is a pocset with no pair of transverse elements. First we need the following lemma.

\begin{lemma}\label{lem:notboth}
	We cannot have both $(\fa_1,\h_1)\leq(\fa_2,\h_2)$ and $(\fa_1,\h_1)^*\leq(\fa_2,\h_2)$.
\end{lemma}
\begin{proof}
	Write $(\fa_1,\h_1)^*=(\fa_1^*,\h_1^*)$.
	If both inequalities in the lemma did hold, then we would have $\fa_1\preceq\fa_2$ and $\fa_1^*\preceq\fa_2$, so $X\preceq\fa_2$ and $\fa_2^*\preceq\fa_2$. 
	If $(\fa_2,\h_2)$ takes the form $(V\h,\h)$ or $(g[x],g\h_0)$, then $\fa_2\subseteq\h_2$ and we get a contradiction since $\h_2^*\npreceq\h_2$ (Lemma \ref{lem:hdeep}).
	If $(\fa_2,\h_2)$ takes the form $(g[x]^*,g\h_0^*)$, then $\fa_2^*=g[x]\preceq g[x]^*$, contradicting Lemma \ref{lem:[x]npreceq[x]*}.
\end{proof}

\begin{lemma}\label{lem:nestedcP}
	$(\cP,\leq,*)$ is a pocset with no pair of transverse elements.
\end{lemma}
\begin{proof}
	Firstly, by Lemma \ref{lem:notboth} and reflexivity of $\leq$ we can never have $(\fa,\h)\leq(\fa,\h)^*$.
	It remains to show that $(\fa_1,\h_1)\leq(\fa_2,\h_2)$ implies $(\fa_2,\h_2)^*\leq(\fa_1,\h_1)^*$, and that $\cP$ has no pair of transverse elements.
	Equivalently, we must show that for all distinct $(\fa_1,\h_1),(\fa_2,\h_2)\in \cP$ exactly two of the following inequalities hold, and that these two inequalities are in the same row.
	\begin{align}\label{rows}
(\fa_1,\h_1)&\leq(\fa_2,\h_2)&(\fa_2,\h_2)^*&\leq(\fa_1,\h_1)^*\\\nonumber
(\fa_1,\h_1)&\leq(\fa_2,\h_2)^*&(\fa_2,\h_2)&\leq(\fa_1,\h_1)^*\\\nonumber
(\fa_1,\h_1)^*&\leq(\fa_2,\h_2)&(\fa_2,\h_2)^*&\leq(\fa_1,\h_1)\\\nonumber
(\fa_1,\h_1)^*&\leq(\fa_2,\h_2)^*&(\fa_2,\h_2)&\leq(\fa_1,\h_1)
	\end{align}

To begin, let's show that if one of the rows of (\ref{rows}) holds then all inequalities in the other rows fail.
Without loss of generality, assume that the first row of (\ref{rows}) holds.
Then $(\fa_1,\h_1)\leq(\fa_2,\h_2)^*$ fails, else we get $(\fa_1,\h_1)\leq(\fa_2,\h_2)^*\leq(\fa_1,\h_1)^*$, contradicting what we said at the start of the proof.
The other inequalities in the second and third rows of (\ref{rows}) fail similarly.
And the inequalities in the fourth row fail because otherwise we could combine them with the inequalities in the first row and deduce that $(\fa_1,\h_1)=(\fa_2,\h_2)$ using antisymmetry of $\leq$.

We will spend the rest of the proof showing that one of the rows of (\ref{rows}) holds.	
Among elements of the form $(V\h,\h)$ in $\cP$ (with $\h\in\cH^\#$), the relation $\leq$ is determined by the projection of these elements to the pocset $(\cH(X),\subseteq)$. Since this pocset has no pairs of transverse elements, we deduce that if $(\fa_1,\h_1),(\fa_2,\h_2)$ are both of the form $(V\h,\h)$ with $\h\in\cH^\#$, then one of the rows of (\ref{rows}) holds.

Henceforth, assume that $(\fa_1,\h_1),(\fa_2,\h_2)$ are not both of the form $(V\h,\h)$ with $\h\in\cH^\#$.
	Without loss of generality, assume $(\fa_1,\h_1)$ is not of this form, so $(\fa_1,\h_1)$ is of the form $(g[x],g\h_0)$ or $(g[x]^*,g\h_0^*)$.
	Since $G$ acts on $\cP$ respecting $\leq$ and $*$, it suffices to consider the case where $(\fa_1,\h_1)$ is of the form $([x],\h_0)$ or $([x]^*,\h_0^*)$.
	And replacing $(\fa_1,\h_1)$ with $(\fa_1,\h_1)^*$ if necessary, we may assume $(\fa_1,\h_1)=([x],\h_0)$.
	We now have three cases to consider:
	
	\begin{enumerate}
		\item\label{item:h=h0} Suppose $\h_2\in\{\h_0,\h_0^*\}$. Replacing $(\fa_2,\h_2)$ with $(\fa_2,\h_2)^*$ if necessary, we may assume that $\h_2=\h_0$.
		Then $(\fa_2,\h_2)=([y],\h_0)$ for some $[y]\in\cM_0$. Since $(\fa_1,\h_1),(\fa_2,\h_2)$ are distinct, we have $[x]\neq[y]$, so the second row of (\ref{rows}) holds by Lemma \ref{lem:[x]leq[y]}.
		
		\item\label{item:hsubseth0} Suppose $\h_2\subsetneq\h_0$ or $\h_2^*\subsetneq\h_0$. Replacing $(\fa_2,\h_2)$ with $(\fa_2,\h_2)^*$ if necessary, we may assume that $\h_2\subsetneq\h_0$.
		By Lemma \ref{lem:hin[x]}, there is a unique $[y]\in\cM_0$ with $V\h_2\subseteq[y]$.
		Let's first show that either the third or fourth row of (\ref{rows}) holds in the case $[x]=[y]$.
		We split into three subcases:
		\begin{enumerate}
			\item Suppose $(\fa_2,\h_2)=(V\h_2,\h_2)$. Then $\fa_2=V\h_2\subseteq[y]=\fa_1$, so the fourth row of (\ref{rows}) holds (and they are inequalities of type \ref{item:preceq}).
			\item Suppose $(\fa_2,\h_2)=(g[z],g\h_0)$, with $g\in G$, $[z]\in\cM_0$.
			Then $\fa_2=g[z]\subseteq V\h_2\subseteq[y]=\fa_1$, so again the fourth row of (\ref{rows}) holds.
			\item Suppose $(\fa_2,\h_2)=(g[z]^*,g\h_0^*)$, with $g\in G$, $[z]\in\cM_0$.
			Let $[y']$ be the $\sim$-class with $V\h_0^*\subseteq g[y']$.
			By Lemma \ref{lem:[y]*preceq} (applied both ways round), we have $g[y']^*\preceq[y]$ and $[y]^*\preceq g[y']$, so
			\begin{equation}\label{g[y']a1}
				(g[y']^*,g\h_0^*)\leq(\fa_1,\h_1)\quad\text{and}\quad(\fa_1,\h_1)^*\leq(g[y'],g\h_0).
			\end{equation}			
			If $[z]=[y']$, then the fourth row of (\ref{rows}) holds by (\ref{g[y']a1}).
			If $[z]\neq[y']$, then
			\begin{equation}\label{a2g[y']}
				(\fa_2,\h_2)^*\leq(g[y']^*,g\h_0^*)\quad\text{and}\quad(g[y'],g\h_0)\leq(\fa_2,\h_2)
			\end{equation}
		by Lemma \ref{lem:[x]leq[y]}, and the third row of (\ref{rows}) holds by combining (\ref{g[y']a1}) and (\ref{a2g[y']}).

		\end{enumerate}		
		Now suppose $[x]\neq[y]$. Then
		\begin{equation}\label{a1[y]}
(\fa_1,\h_1)\leq([y],\h_0)^* \quad\text{and}\quad ([y],\h_0)\leq(\fa_1,\h_1)^*
		\end{equation}  by Lemma \ref{lem:[x]leq[y]}.
		By what we showed above for the case $[x]=[y]$, we know that either the third or fourth row of (\ref{rows}) holds for the pair $([y],\h_0),(\fa_2,\h_2)$, namely we must be in one of the following two subcases:
		\begin{enumerate}[label={(\roman*)}]
			\item Suppose $([y],\h_0)^*\leq(\fa_2,\h_2)$ and $(\fa_2,\h_2)^*\leq([y],\h_0)$.
			Combining these inequalities with (\ref{a1[y]}) we deduce that the first row of (\ref{rows}) holds for $(\fa_1,\h_1),(\fa_2,\h_2)$.
			\item Suppose $([y],\h_0)^*\leq(\fa_2,\h_2)^*$ and $(\fa_2,\h_2)\leq([y],\h_0)$.
			Combining these inequalities with (\ref{a1[y]}) we deduce that the second row of (\ref{rows}) holds for $(\fa_1,\h_1),(\fa_2,\h_2)$.
		\end{enumerate}
				
		\item Suppose $\h_0\subsetneq\h_2$ or $\h_0\subsetneq\h_2^*$. Replacing $(\fa_2,\h_2)$ with $(\fa_2,\h_2)^*$ if necessary, we may assume that $\h_0\subsetneq\h_2$.
		Again, we split into three subcases:
		\begin{enumerate}
			\item If $(\fa_2,\h_2)=(V\h_2,\h_2)$, then $[x]\subseteq\h_0\subseteq\h_2$, so the first row of (\ref{rows}) holds (and they are both inequalities of type \ref{item:preceq}).
			\item If $(\fa_2,\h_2)=(g[z],g\h_0)$, with $g\in G$ and $[z]\in\cM_0$, then up to the action of $G$ (and with the roles of $(\fa_1,\h_1),(\fa_2,\h_2)$ reversed) this is the same as case \ref{item:hsubseth0}.
			\item Finally, if $(\fa_2,\h_2)=(g[z]^*,g\h_0^*)$, with $g\in G$ and $[z]\in\cM_0$, then $[x]\subseteq \h_0\subseteq \h_2\preceq g[z]^*$, so the first row of (\ref{rows}) holds (and they are both inequalities of type \ref{item:preceq}).\qedhere
		\end{enumerate}	
	\end{enumerate}
\end{proof}

In the next four lemmas we prove that $\cP$ is a discrete pocset (see Definition \ref{defn:pocset}).

\begin{lemma}\label{lem:hina}
	If $(\fa_1,\h_1)\leq(\fa_2,\h_2)$ with $\hh_1\neq\hh_2$ then $V\hh_1\subseteq\fa_2$.
\end{lemma}
\begin{proof}
	We split into three cases depending on the form of $(\fa_2,\h_2)$.
	\begin{enumerate}
		\item\label{item:Vh2} Suppose $(\fa_2,\h_2)=(V\h_2,\h_2)$ with $\h_2\in\cH^\#$.
		If $\hh_1\subseteq\h_2$ then we are done, so suppose  $\hh_1\subseteq\h_2^*$.
		As $\fa_1\preceq\fa_2=V\h_2$ and $\hh_1\neq\hh_2$, we reduce to two possibilities for $(\fa_1,\h_1)$:
		\begin{enumerate}
			\item Suppose $(\fa_1,\h_1)=(V\h_1,\h_1)$ with $\h_2\subsetneq\h_1$.
			Then $\fa_2=V\h_2\subseteq V\h_1=\fa_1$ and $\h_1\not\subseteq\h_2$, contradicting $(\fa_1,\h_1)\leq(\fa_2,\h_2)$.
			\item Suppose $(\fa_1,\h_1)=(g[x],g\h_0)$ with $\h_2\subsetneq g\h_0$.
			By Lemma \ref{lem:hin[x]} we know that $V\h_2\subseteq g[y]$ for some $[y]\in\cM_0$.
			If $[x]\neq[y]$ then $g[x]=\fa_1\preceq\fa_2=V\h_2\subseteq g[y]\subseteq g[x]^*$, contradicting Lemma \ref{lem:[x]npreceq[x]*}.
			If $[x]=[y]$ then $\fa_2=V\h_2\subseteq g[y]=\fa_1$ and we contradict $(\fa_1,\h_1)\leq(\fa_2,\h_2)$ (noting $\h_1\not\subseteq\h_2$).
		\end{enumerate}
		
		\item Suppose $(\fa_2,\h_2)=(g_2[x_2],g_2\h_0)$.
		If $\hh_1\subseteq\h_2^*$ then we can obtain a contradiction by the same argument as in case \ref{item:Vh2}, so we may suppose $\hh_1\subseteq\h_2=g_2\h_0$.
		By Lemma \ref{lem:hin[x]} we have $V\hh_1\subseteq g_2[y_2]$ for some $[y_2]\in\cM_0$.
		If $[x_2]=[y_2]$ we are done, so suppose $[x_2]\neq[y_2]$.
		\begin{enumerate}
			\item Suppose $\h_1\subseteq g_2\h_0$.
			Then $V\h_1\subseteq g_2[y_2]$, which is disjoint from $\fa_2=g_2[x_2]$, so any path from $\fa_1\cap\h_1$ to $\fa_2=g_2[x_2]$ must first cross $\hh_1$.
			But $\fa_1\cap\h_1\not\preceq\hh_1$ (indeed this holds for any $(\fa,\h)\in\cP$), so $\fa_1\not\preceq\fa_2$, contradicting $(\fa_1,\h_1)\leq(\fa_2,\h_2)$.
			\item Suppose $\h_1^*\subseteq g_2\h_0$.
			If $V\h_2^*=g_2 V\h_0^*\subseteq\fa_1$ then we contradict $\fa_1\preceq\fa_2$, hence $(\fa_1,\h_1)$ must take the form $(g_1[x_1],g_1\h_0)$, with $g_2 V\h_0^*\not\subseteq g_1[x_1]$.
			By Lemma \ref{lem:hin[x]} we have $g_2 V\h_0^*\subseteq g_1[y_1]$ for some $[x_1]\neq[y_1]\in\cM_0$.
			But then Lemma \ref{lem:[y]*preceq} implies that $g_2[y_2]^*\preceq g_1[y_1]$, so $g_1[x_1]=\fa_1\preceq\fa_2=g_2[x_2]\subseteq g_2[y_2]^*\preceq g_1[y_1]\subseteq g_1[x_1]^*$.
			This contradicts Lemma \ref{lem:[x]npreceq[x]*}.
		\end{enumerate}
		
		\item Suppose $(\fa_2,\h_2)=(g_2[x_2]^*,g_2\h_0^*)$.
		If $\hh_1\subseteq\h_2=g_2\h_0^*$ then we are done, so suppose $\hh_1\subseteq g_2\h_0$.
		By Lemma \ref{lem:hin[x]} we have $V\hh_1\subseteq g_2[y_2]$ for some $[y_2]\in\cM_0$.
		If $[x_2]\neq[y_2]$ then $V\hh_1\subseteq g_2[y_2]\subseteq g_2[x_2]^*=\fa_2$ and we are done.
		So suppose $[x_2]=[y_2]$.
		\begin{enumerate}
			\item Suppose $\h_1\subseteq g_2\h_0$. Then $\h_1\subseteq g_2[x_2]$, so any path from $\fa_1\cap\h_1$ to $\fa_2=g_2[x_2]^*$ must first cross $\hh_1$. But $\fa_1\cap\h_1\not\preceq\hh_1$ (indeed this holds for any $(\fa,\h)\in\cP$), so $\fa_1\not\preceq\fa_2$, contradicting $(\fa_1,\h_1)\leq(\fa_2,\h_2)$.
			\item Suppose $\h_1^*\subseteq g_2\h_0$.
			If $V\h_1-V\hh_1\subseteq\fa_1$ then $\fa_2=g_2[x_2]^*\subseteq\fa_1$, and, combined with $\h_1\not\subseteq\h_2$, this contradicts $(\fa_1,\h_1)\leq(\fa_2,\h_2)$.
			So $V\h_1-V\hh_1\not\subseteq\fa_1$, and $(\fa_1,\h_1)$ takes the form $(g_1[x_1],g_1\h_0)$.
			By Lemma \ref{lem:hin[x]} there is $[y_1]\in\cM_0$ with $V\h_2=g_2 V\h_0^*\subseteq g_1[y_1]$.
			By Lemma \ref{lem:[y]*preceq}, we have $\fa_2=g_2[x_2]^*\preceq g_1[y_1]$.
			If $[x_1]=[y_1]$ then $\fa_2\preceq\fa_1$, and we again contradict $(\fa_1,\h_1)\leq(\fa_2,\h_2)$.
			So $[x_1]\neq[y_1]$.
			But then $g_1[x_1]=\fa_1\preceq\fa_2\preceq g_1[y_1]\subseteq g_1[x_1]^*$ and we contradict Lemma \ref{lem:[x]npreceq[x]*}.\qedhere		
		\end{enumerate}
	\end{enumerate}
\end{proof}

\begin{lemma}\label{lem:einhhnbhd}
	There is a constant $K>0$ such that, for any $(\fa,\h)\in\cP$, any edge $e$ going from $\fa$ to $VX-\fa$ is contained in the $K$-neighborhood of $\hh$.
\end{lemma}
\begin{proof}
	Let $K\geq1$ be such that $\delta$ is contained in the $K$-neighborhood of $\hh_0$.
	If $(\fa,\h)$ takes the form $(V\h,\h)$ with $\h\in\cH^\#$ then $e$ must go from $\hh$ to $X-\h$, so $e$ is certainly in the $K$-neighborhood of $\hh$.
	If $(\fa,\h)$ takes the form $(g[x],g\h_0)$, then $e$ either goes from $g\hh_0$ to $X-g\h_0$ or it goes from $g[x]$ to $g[y]$ for some $x\nsim y\in V\h_0$.
	In the former case, $e$ is certainly in the $K$-neighborhood of $\hh=g\hh_0$.
	In the latter case, $e$ must belong to a set $gg_0\delta$ with $g_0\in G_0$, so $e$ is in the $K$-neighborhood of $\hh=g\hh_0$ by definition of $K$.
	If $(\fa,\h)$ takes the form $(g[x]^*,g\h_0^*)$ then $\fa$ is the complement in $VX$ of $g[x]$, so the argument works the same as the previous case.
\end{proof}

\begin{lemma}\label{lem:eachh3}
Let $(\fa_1,\h_1)\leq(\fa_2,\h_2)$ in $\cP$. For each $\h_3\in\cH(X)$ there is at most one $(\fa_3,\h_3)\in\cP$ with $(\fa_1,\h_1)\leq(\fa_3,\h_3)\leq(\fa_2,\h_2)$.
\end{lemma}
\begin{proof}
	If $\h_3\in\cH^\#$ then we must have $\fa_3=V\h_3$.
	Now suppose $\h_3=g\h_0$.
	Suppose there are two choices for $\fa_3$ that work.
	This means there are distinct $[x],[y]\in\cM_0$ with $(g[x],g\h_0),(g[y],g\h_0)$ both greater than $(\fa_1,\h_1)$ and both less than $(\fa_2,\h_2)$.
	By Lemma \ref{lem:[x]leq[y]}, we have $(g[x],g\h_0)\leq(g[y]^*,g\h_0^*)=(g[y],g\h_0)^*$.
	Also $(\fa_1,\h_1)\leq(g[y],g\h_0)$ implies $(g[y],g\h_0)^*\leq(\fa_1,\h_1)^*$.
	Putting this together we get $(\fa_1,\h_1)\leq(g[x],g\h_0)\leq(\fa_1,\h_1)^*$, which contradicts $\cP$ being a pocset (Definition \ref{defn:pocset}\ref{item:AA*}).
	So there is at most one choice for $\fa_3$ that works.
	Finally, the case $\h_3=g\h_0^*$ can be reduced to the case $\h_3=g\h_0$ by applying the involution $*$ of $\cP$.
\end{proof}

\begin{lemma}\label{lem:discrete}
	$\cP$ is discrete.
\end{lemma}
\begin{proof}
	Let $(\fa_1,\h_1)\leq(\fa_2,\h_2)$ in $\cP$. We must show that there are only finitely many $(\fa_3,\h_3)\in\cP$ with $(\fa_1,\h_1)\leq(\fa_3,\h_3)\leq(\fa_2,\h_2)$. 
	By Lemma \ref{lem:eachh3}, we may restrict to those $(\fa_3,\h_3)$ with $\hh_3\neq\hh_1,\hh_2$.	
	Fix a path $\gamma$ in $X$ from $\hh_1$ to $\hh_2$.
	For each $(\fa_3,\h_3)$, we have $(\fa_1,\h_1)\leq(\fa_3,\h_3)$ and $(\fa_2,\h_2)^*\leq(\fa_3,\h_3)^*$, so Lemma \ref{lem:hina} tells us that $V\hh_1\subseteq\fa_3$ and $V\hh_2\subseteq \fa_3^*$ (where we write $(\fa_3,\h_3)^*=(\fa_3^*,\h_3^*)$).
	Since $\hh_3\neq\hh_2$ and $\fa_3^*\subseteq (VX-\fa_3)\cup\hh_3$, we also have $V\hh_2\subseteq VX-\fa_3$.
	As a result, the path $\gamma$ contains an edge $e$ that goes from $\fa_3$ to $VX-\fa_3$.
	By Lemma \ref{lem:einhhnbhd}, this edge $e$ is contained in the $K$-neighborhood of $\hh_3$ (where $K>0$ is a fixed constant).
	By local finiteness of $X$, and the fact that the collection of walls $(\hh)$ is disjoint, we see that there are only finitely many possibilities for $\h_3$.
	Finally, by Lemma \ref{lem:eachh3}, we deduce that there are only finitely many possibilities for $(\fa_3,\h_3)$.
\end{proof}

\subsection{The tree $T'$}\label{subsec:T'}

Let $T'$ be the cubing of $\cP$.
We know that $\cP$ is discrete by Lemma \ref{lem:discrete} and it has no pair of transverse elements by Lemma \ref{lem:nestedcP}, hence Proposition \ref{prop:cubingtree} tells us that the cubing $T'$ is well-defined and is a tree.
Moreover, the action of $G$ on $\cP$ induces an action of $G$ on $T'$.
The edges in $T'$ correspond to pairs $\{(\fa,\h),(\fa,\h)^*\}\subseteq\cP$, and, since $(\fa,\h)^*$ always takes the form $(\fa',\h^*)$, an edge inversion in $T'$ would mean there is an element $g\in G$ with $g\h=\h^*$.
But that would give rise to an edge inversion in $T$, which doesn't happen, hence $G$ acts on $T'$ without edge inversions, i.e. $G\acts T'$ is a splitting.
We may assume that the action $G\acts T'$ is minimal since we can always replace $T'$ with a minimal $G$-invariant subtree.
The description of edges in $T'$ also makes it clear that every edge stabilizer of $T'$ is contained in an edge stabilizer of $T$.

\begin{lemma}
	Every vertex stabilizer of $T'$ is contained in a vertex stabilizer of $T$.
\end{lemma}
\begin{proof}
	We will prove this by constructing a $G$-equivariant map $\phi:VT'\to VT$.
	As in Proposition \ref{prop:cubing}, we will view the vertices of $T$ as DCC ultrafilters on $\cH(X)$ (which is isomorphic to the pocset $\cH(T)$) and the vertices of $T'$ as DCC ultrafilters on $\cP$.
	Given $v\in VT'$, we then define
	\begin{align*}
		\phi(v):=&\{\h\in\cH^\#\mid(V\h,\h)\in v\}\\
		&\cup\{g\h_0\mid g\in G,\,(g[x],g\h_0)\in v\text{ for some }[x]\in\cM_0\}\\
		&\cup\{g\h_0^*\mid g\in G,\,(g[x]^*,g\h_0^*)\in v\text{ for all }[x]\in\cM_0\}.
	\end{align*}
	Note that the properties required for $g\h_0$ and $g\h_0^*$ to be in $\phi(v)$ only depend on the halfspaces $g\h_0$ and $g\h_0^*$, not on the element $g\in G$.
	The map $\phi:VT'\to VT$ is $G$-equivariant by definition of the $G$-action on $\cP$.
	Following Definition \ref{defn:ultrafilter}, we must show that $\phi(v)$ is complete, consistent, and DCC.
	
	Let's first prove completeness of $\phi(v)$.
	For $\h\in\cH^\#$, exactly one of $(V\h,\h)$ and $(V\h^*,\h^*)$ is in $v$ by completeness of $v$, so exactly one of $\h$ and $\h^*$ is in $\phi(v)$.
	For $g\in G$ and $[x]\in\cM_0$, exactly one of $(g[x],g\h_0)$ and $(g[x]^*,g\h_0^*)$ is in $v$ by completeness of $v$, so for each $g\in G$ we have exactly one of $g\h_0$ and $g\h_0^*$ in $\phi(v)$.
	
	Before showing that $\phi(v)$ is consistent and DCC, we make the following observation, which follows from Lemmas \ref{lem:hin[x]} and \ref{lem:[y]*preceq}: if $\h_1\subsetneq\h_2$ in $\cH(X)$ then
	\begin{equation}\label{leq}
		(\fa_1,\h_1)\leq(\fa_2,\h_2)
	\end{equation}  
	in $\cP$, where
	\begin{enumerate}
		\item\label{item:h1} if $\h_1\in\cH^\#$ then $\fa_1=V\h_1$,
		\item\label{item:h1h0} if $\h_1=g_1\h_0$ then $\fa_1=g_1[x]$ and (\ref{leq}) holds for all choices of $[x]\in\cM_0$,
		\item\label{item:h1h0*} if $\h_1=g_1\h_0^*$ then $\fa_1=g_1[x]^*$ for the unique $[x]\in\cM_0$ with $V\h_2^*\subseteq g_1[x]$,
		\item\label{item:h2} if $\h_2\in\cH^\#$ then $\fa_2=V\h_2$,
		\item\label{item:h2h0} if $\h_2=g_2\h_0$ then $\fa_2=g_2[y]$ for the unique $[y]\in\cM_0$ with $V\h_1\subseteq g_2[y]$,
		\item\label{item:h2h0*} if $\h_2=g_2\h_0^*$ then $\fa_2=g_2[y]^*$ and (\ref{leq}) holds for all choices of $[y]\in\cM_0$.
	\end{enumerate}
	
	We now prove consistency of $\phi(v)$.
	If $\h_1\subsetneq\h_2$ and $\h_1\in\phi(v)$ then $(\fa_1,\h_1)\in v$, where $(\fa_1,\h_1)$ is given by the appropriate case from \ref{item:h1}--\ref{item:h1h0*} (and in case \ref{item:h1h0} we only have $(g_1[x],g_1\h_0)\in v$ for one $[x]\in\cM_0$). Consistency of $v$ then implies that $(\fa_2,\h_2)\in v$, where $(\fa_2,\h_2)$ is given by the appropriate case from \ref{item:h2}--\ref{item:h2h0*} (and in case \ref{item:h2h0*} we have $(g_2[y]^*,g_2\h_0^*)\in v$ for all $[y]\in\cM_0$). It follows that $\h_2\in\phi(v)$.
	
	Finally, we prove that $\phi(v)$ is DCC.
	Suppose for contradiction that we have an infinite descending chain $\h_1\supsetneq\h_2\supsetneq\h_3\supsetneq\dots$ of distinct elements in $\phi(v)$.
	Then we can turn this into an infinite descending chain
	\begin{equation}\label{chain}
	(\fa_1,\h_1)\geq(\fa_2,\h_2)\geq(\fa_3,\h_3)\geq\dots
	\end{equation}
	in $v$, where
	\begin{enumerate}
		\item if $\h_i\in\cH^\#$ then $\fa_i=V\h_i$,
		\item if $\h_i=g\h_0$ then $\fa_i=g[x]$ for the unique $[x]\in\cM_0$ with $V\h_{i+1}\subseteq g[x]$,
		\item if $\h_i=g\h_0^*$ then $\fa_i=g[x]^*$ for the unique $[x]\in\cM_0$ with $V\h_{i-1}^*\subseteq g[x]$ (or any $[x]\in\cM_0$ if $i=1$).
	\end{enumerate}
	(\ref{chain}) is a chain for the same reason that the inequality (\ref{leq}) holds.
	But of course this contradicts the fact that $v$ is DCC.
\end{proof}

Property \ref{item:T'toTe} in Theorem \ref{thm:smallerstabs} is given by the following lemma.

\begin{lemma}\label{lem:ET'toVTe}
	There is an injective $G$-equivariant map $\tau:ET'\to\sqcup_e VT_e$.
\end{lemma}
\begin{proof}
	Recall that we have a wallspace $(V\h_0,\cP_0)$ with cubing $T_0$, and an associated map $\lambda: V\h_0\to VT_0$. Moreover, it follows from the constructions that $\lambda$ is $G_0$-equivariant.
	We observed that the $\sim$-classes $[x]$ are precisely the non-empty fibers of $\lambda$, so in particular we get a well-defined injective $G_0$-equivariant map $\cM_0\to VT_0$ given by $[x]\mapsto\lambda(x)$. 
	
	Now recall the construction of the trees $T_e$ ($e\in ET$) in Subsection \ref{subsec:Te}, and the action of $G$ on $\sqcup_e T_e$.
	With $e_0$ being the edge corresponding to the halfspace $\h_0$, the tree $T_{e_0}$ is identified with the tree $T_0$.
	We now extend the map $\cM_0\to VT_0$ constructed above to a map $\tau:ET'\to\sqcup_e VT_e$.
	We identify the edges of $T'$ with pairs $\{(\fa,\h),(\fa,\h)^*\}\subseteq\cP$, so we have two types of edges to consider.
	\begin{itemize}
		\item For $\h\in\cH^\#$, we have an edge $e'=\{(V\h,\h),(V\h^*,\h^*)\}\in ET'$ and also an edge $e\in ET$.
		In this case $e\notin G\cdot e_0$, so the tree $T_e$ is just a single vertex, and we define $\tau(e')$ to be this vertex.
		The maps $\cH^\#\to ET$ and $\cH^\#\to ET'$ are $G$-equivariant, so $\tau(ge')=g\cdot\tau(e')$ for all $g\in G$.
		
		\item Recall that we have a left transversal $\{g_i\mid i\in\Omega\}$ of $G_0$ in $G$, so the second type of edge in $T'$ can be written in the form
		$$e'=\{(g_i[x],g_i\h_0),(g_i[x]^*,g_i\h_0^*)\},$$
		with $i\in\Omega$ and $[x]\in\cM_0$.
		We then define
		$$\tau(e'):=(\lambda(x),g_i G_0)\in T_0\times\{g_i G_0\}=:T_{g_i e_0}.$$
		To check $G$-equivariance, take $g\in G$, and let $gg_i=g_jg_0$ with $j\in\Omega$ and $g_0\in G_0$.
		The definition of $G\acts\sqcup_e T_e$ tells us that
		$$g\cdot(\lambda(x),g_i G_0)=(g_0\lambda(x),g_jG_0).$$
		Meanwhile,
		\begin{align*}
			\tau(ge')&=\tau(\{(gg_i[x],gg_i\h_0),(gg_i[x]^*,gg_i\h_0^*)\})\\
			&=\tau(\{(g_jg_0[x],g_jg_0\h_0),(g_jg_0[x]^*,g_jg_0\h_0^*)\})\\
			&=\tau(\{(g_j[g_0x],g_j\h_0),(g_j[g_0x]^*,g_j\h_0^*)\})\\
			&=(\lambda(g_0x),g_jG_0)\\
			&=(g_0\lambda(x),g_jG_0).
		\end{align*}
		So $\tau(ge')=g\cdot\tau(e')$, as required.
	\end{itemize}
	Finally, the map $\tau:ET'\to\sqcup_e VT_e$ is injective because the map $\cM_0\to VT_0$ given by $[x]\mapsto\lambda(x)$ is injective. 
\end{proof}

It follows from Lemma \ref{lem:ET'toVTe} and property \ref{item:Te} of Theorem \ref{thm:smallerstabs} that each edge stabilizer of $G\acts T'$ is equal to a vertex stabilizer from one of the splittings $G_e\acts T_e$.
In particular, the edge stabilizers of $G\acts T'$ are finitely generated by Lemma \ref{lem:fgvertex}.
Finally, property \ref{item:nofixT'} of Theorem \ref{thm:smallerstabs} is given by the following lemma.

\begin{lemma}
	If no vertex stabilizer of $T$ fixes an edge in $T$, then there is a vertex $v_0\in VT$ such that $G_{v_0}$ has no fixed point in $T'$.
\end{lemma}
\begin{proof}
	Recall that $e_0\in ET$ is the edge corresponding to the halfspace $\h_0$, in particular $G_0=G_{e_0}$.
	The vertex $v_0\in VT$ will be the endpoint of $e_0$ that is contained in the halfspace $p(\h_0)\in\cH(T)$.
	The vertex stabilizer $G_{v_0}$ fixes no edges in $T$, so there exists $g\in G_{v_0}-G_{e_0}$.
	Then $g\h_0^*\subsetneq\h_0$.
	By Lemma \ref{lem:[y]*preceq}, there exist $[x],[y]\in\cM_0$ with $g[y]^*\preceq[x]$, so $(g[y]^*,g\h_0^*)\leq([x],\h_0)$ in $\cP$.
	Recall that we have a $G_0$-equivariant map $\cM_0\to T_0$ (as in the proof of Lemma \ref{lem:ET'toVTe}), and $G_0$ has no fixed point in $T_0$ by Lemma \ref{lem:T0nofixed}, so the $G_0$-orbits in $\cM_0$ are all infinite.
	In particular, there exists $g_0\in G_0$ with $g_0[x]\neq[y]$.
	Then $(g_0[x],\h_0)\leq([y]^*,\h_0^*)$ by Lemma \ref{lem:[x]leq[y]}.
	Putting this all together, we get
	$$gg_0\cdot([x],\h_0)=(gg_0[x],g\h_0)\leq(g[y]^*,g\h_0^*)\leq([x],\h_0).$$
	Therefore $gg_0\in G_{v_0}$ acts hyperbolically on $T'$, translating along the axis that contains the edges corresponding to $\{(gg_0)^i\cdot([x],\h_0),(gg_0)^i\cdot([x]^*,\h_0^*)\}$ for $i\in\Z$.
	As a result, $G_{v_0}$ has no fixed point in $T'$.
\end{proof}

\section{Groups $G$ that are simply connected at infinity; have $H^2(G,\mathbb ZG)=\{0\}$ }\label{sec:scatinfty}
We begin this section with basic definitions. The notion of simple connectivity at infinity for groups and spaces is a classical and well studied notion. We use \cite{Geoghegan08} as a basic reference.

\begin{definition} 
The space $X$ is {\it simply connected at infinity}  if for any compact set $C$ there is a compact set $D$ such that loops $\alpha:[0,1]\to X-D$ are homotopically trivial in $X-C$. This means there is a homotopy $H:[0,1]\times [0,1]\to X-C$ such that $H(t,0)=\alpha(t)$, $H(0,t)=H(1,t)=H(t,1)=H(0,0)$ for all $t\in [0,1]$.
\end{definition}

\begin{definition} 
The space $X$ is {\it 1-acyclic at infinity} (has {\it pro-finite first homology at infinity}) if for any compact set $C$ there is a compact set $D$ such that the image of $H_1(X-D)$ in $H_1(X-C)$ under the homomorphism induced by the inclusion of $X-D$ into $X-C$ is  trivial (respectively, finite).  
A loop $\alpha$ in $X-C$ is homologically trivial if there is an orientable 2-manifold $M$ bounded by an embedded loop $\alpha'$ and a continuous map $H:M\to X-C$ such that $H$ restricted to $\alpha'$ maps (in the obvious way) to $\alpha$ (see Proposition 12.8, \cite{GH81}).
\end{definition}

\begin{definition}\label{n-equiv} 
A (proper) cellular map $f:X\to Y$ between CW-complexes is a {\it (proper) $k$-equivalence} if there is a (proper) cellular map $g:Y\to X$ such that $gf:X^{k-1}\to X$ is (properly) homotopic to the inclusion $X^{k-1}\to X$ and $fg:Y^{k-1}\to Y$ is (properly) homotopic to the inclusion $Y^{k-1}\to Y$.  
\end{definition}

\begin{theorem} [See the first paragraph of \S 16.5,\cite{Geoghegan08}] \label{P1E} 
Suppose $X$ and $Y$ are finite connected CW-complexes with $\pi_1(X)$ isomorphic to $\pi_1(Y)$. Then there is a 2-equivalence between $X$ and $Y$, and hence (by lifting) a proper 2-equivalence between the 2-skeletons  of the universal covers of $X$ and $Y$.
\end{theorem}

\begin{theorem} [Proposition 16.2.3 and 16.1.14, \cite{Geoghegan08}] \label{2-equiv}
If $X$ and $Y$ are locally finite, connected CW-complexes and there is a proper 2-equivalence $f:X\to Y$, then 
$X$ is simply connected at infinity, 1-acyclic at infinity, or has pro-finite first homology at infinity if and only if $Y$ does as well.
\end{theorem}

\begin{definition} \label{SCIgp}
If $G$ is a finitely presented group and $X$ is some (equivalently any) 
connected finite complex with $\pi_1(X)$ isomorphic to $G$, then $G$ is {\it simply connected at infinity}, is {\it 1-acyclic at infinity}, or has {\it pro-finite first homology at infinity} if the same holds for the universal cover of $X$. 
\end{definition}

The next result follows directly from parts (ii) and (iii) of the main corollary of \cite{GeogheganMihalik86} in the special case with $n=2$.  In particular, if the finitely presented group $G$ is simply connected at infinity, then it is 1-acyclic at infinity and if $G$ is 1-acyclic at infinity, then the first homology at infinity of $G$ is pro-finite. 

\begin{theorem} [\cite{GeogheganMihalik86}] \label{GMpro} If $G$ is a finitely presented group, then $H^2(G,\mathbb ZG)=\{0\}$ if and only if the first homology at infinity of $G$ is pro-finite. 
\end{theorem}

Finally we mention the following:
\begin{theorem} [Theorem 18.2.11, \cite{Geoghegan08}] \label{pnequiv-qi}
If $G$ and $H$ are finitely presented quasi-isometric groups then $G$ is simply connected at infinity, is 1-acyclic at infinity, or has pro-finite first homology at infinity if and only if $H$ does as well. 
\end{theorem}

\section {Proof of Theorem \ref{GeoSplit}}\label{sec:acyclicthm} 

\theoremstyle{plain}
\newtheorem*{GeoSplit}{Theorem \ref{GeoSplit}}
\begin{GeoSplit}
	Suppose $X$ is a locally finite CW-complex, and $X_1$ and $X_2$ are connected one-ended subcomplexes of $X$ such that $X_1\cup X_2=X$. If $K$ is a finite subcomplex of $X$ (possibly empty) such that  $(X_1\cap X_2)-K$ has more than one unbounded component, then $X$ does not have pro-finite first homology at infinity. In particular, $X$ is not simply connected at infinity.
\end{GeoSplit}
\begin{proof}
	Assume for contradiction that $X$ has pro-finite first homology at infinity. Let $X_0=X_1\cap X_2$. 
	Let $D$ be a finite subcomplex of $X$ containing $K$ such that for any loop $\tau$ in $X-D$, there is $n>0$ such that $\tau^n$ is homologically trivial in $X-K$. Since  both $X_1$ and $X_2$ are one-ended, there is a finite subcomplex $E$ of $X$ containing $D$ such that any two vertices of $X_0-E$ can be joined by an edge path in $X_1-D$ and by an edge path in $X_2-D$. 
	Choose vertices $v$ and $w$ in $X_0-E$ such that $v$ and $w$ cannot be joined by an edge path in $X_0-K$ (so $v\in V$ and $w\in W$ where $V$ and $W$ are distinct components of $X_0-K$). Let $\alpha$ be an edge path in $X_1-D$ from $v$ to $w$ and let $\beta$ be an edge path in $X_2-D$ from $w$ to $v$. 
	The concatenation, $\alpha*\beta$, is a loop in $X-D$. By our earlier assumption, $(\alpha*\beta)^n$ is homologically trivial in $X-K$. Let $M$ be an orientable 2-manifold bounded by an embedded loop $\alpha'_1*\beta_1'*\ldots* \alpha_n'*\beta_n'$ and let $H:M\to X-K$ be a continuous map such that $H$ restricted to $\alpha_i'$ and $\beta_i'$ maps (in the obvious way) to the edge paths $\alpha$ and $\beta$ respectively. Let $v_i'$ and $w_i'$ be the initial and terminal points of $\alpha_i'$ (See Figure \ref{Fig1}).
	
	\begin{figure}
		\vbox to 3in{\vspace {-2in} \hspace {-.3in}
			\hspace{-1 in}
			\includegraphics[scale=1]{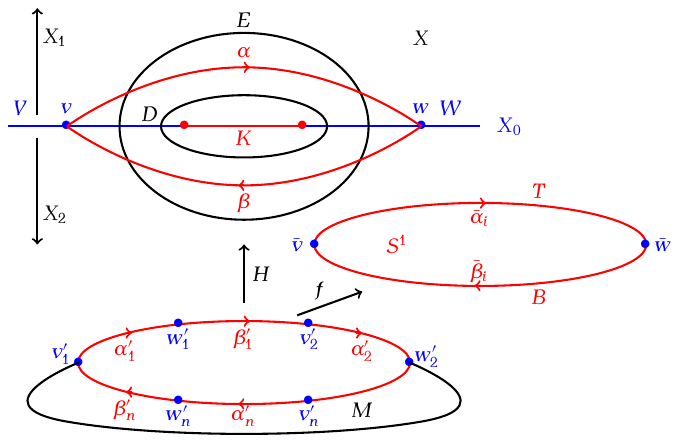}
			\vss }
		\vspace{-.5in}
		\caption{The Relevant Maps}
		\label{Fig1}
	\end{figure}

	\begin{lemma} \label{VW} The sets $H^{-1}(W)$ and $H^{-1}(X_0-W)$ are disjoint closed subsets of $M$ containing $\{w_1',\ldots  ,w_m'\}$ and $\{v_1',\ldots, v_n'\}$ respectively. 
	\end{lemma}
	\begin{proof}
		First note that $H^{-1}(X_0)$ is a closed subset of $M$.  As $X_0$ is locally connected, each component of $X_0-K$ is open in $X_0$. In particular, $X_0-W$ is closed in $X_0$ and hence closed in $X$. Then $H^{-1}(X_0-W)$ is closed in $M$. The set $W\cup K$ is closed in $X_0$ and in $X$. As the image of $H$ misses $K$, $H^{-1}(W) =H^{-1}(W\cup K)$ and is closed in $M$. 
	\end{proof} 
	
	Consider the circle $S^1$ with diametrically opposite vertices $\bar v$ and $\bar w$. Let $T$ and $B$ be the two closed arcs in $S^1$ such that $T\cup B=S^1$ and $T\cap B=\{\bar v,\bar w\}$. Map the (disjoint) closed sets $H^{-1}(W)$ and $H^{-1} (X_0-W)$ in $H^{-1}(X_1)$ to $\{\bar w\}$ and $\{\bar v\}$ respectively. Extend this map by Tietze's extension theorem to a continuous map $f_1:H^{-1}(X_1)\to T$. Similarly define $f_2:H^{-1}(X_2)\to B$ (with $f_2(H^{-1}(W))=\{\bar w\}$ and $f_2(H^{-1}(X_0-W))=\{\bar v\}$). Since $f_1$ and $f_2$ agree on the closed set $H^{-1}(X_1)\cap H^{-1} (X_2)=H^{-1}(X_0)= H^{-1}(W)\cup H^{-1} (X_0-W)$, we have a continuous function $f:M\to S^1$ that agrees with $f_1$ and $f_2$ on $H^{-1}(X_1)$ and $H^{-1}(X_2)$ respectively. Note that for each $i$, $f\alpha_i'=f_1\alpha_i'=\bar \alpha_i$ is a path in $T$ from $\bar v$ to $\bar w$ and $f\beta_i'=f_2\beta_i'=\bar \beta_i$ is a path in $B$ from $\bar w$ to $\bar v$. Furthermore, for each $i$, the homotopy class of loop $(\bar \alpha_i, \bar \beta_i)$ at $\bar v$ generates $\pi_1(S_1,\bar v)=\mathbb Z$. The contradiction arises since the homotopy class of the loop $\alpha_1'* \beta_1'*\ldots*\alpha_n'*\beta_n'$ is a commutator in $\pi_1(M, v')$ and so maps to the trivial element under the homomorphism $f_\ast:\pi_1(M,v')\to \pi_1(S^1,\bar v)$. But instead this class is mapped by $f_\ast$ to  the class of $\bar \alpha_1*\bar \beta_1*\ldots* \bar \alpha_n* \bar \beta_n$, a non-trivial element of $\pi_1(S^1,\bar v)$. 
\end{proof}

We now prove Corollary \ref{Acyclic}.

\theoremstyle{plain}
\newtheorem*{Acyclic}{Corollary \ref{Acyclic}}
\begin{Acyclic}
	Let $G\acts T$ be a non-trivial splitting with minimal action, with $G$ one-ended and finitely presented. Suppose the edge stabilizers are finitely generated, and suppose there is some edge stabilizer $G_e$ with more than one end.
	If the two halfspaces of $G\acts T$ associated to $e$ are one-ended then $H^2(G,\mathbb ZG)\ne \{0\}$.
	\end{Acyclic}
\begin{proof}
	Following Definition \ref{defn:halfspaceT}, fix an orbit map $f:G\to VT$ and fix a finite generating set $S$ for $G$. Write $X$ for the Cayley graph of $G$ with respect to $S$. The halfspaces in $T$ bounded by the edge midpoint $\hat{e}$ have preimages in $G$ that induce halfspaces $\h_1$ and $\h_2$ in $X$. As discussed in Definition \ref{defn:halfspaceT}, $\h_1$ and $\h_2$ are connected subgraphs of $X$ for suitable choice of the finite generating set $S$. By hypothesis, $\h_1$ and $\h_2$ are one-ended.
	
	The halfspaces $\h_1$ and $\h_2$ are disjoint, but contain all the vertices of $X$. Let $X_1$ and $X_2$ be the $L$-neighborhoods of $\h_1$ and $\h_2$ respectively, with $L\geq 1$ an integer. Then $X=X_1\cup X_2$, and $X_1$ and $X_2$ are one-ended (they are quasi-isometric to $\h_1$ and $\h_2$).
	
	We now claim that $X_1\cap X_2$ is connected for large enough $L$, and that $G_e$ acts cocompactly on $X_1\cap X_2$.
	Let $\Sigma_e$ be the set of edges $(x_1,x_2)$ in $X$ with $x_i\in\h_i$. For $(x_1,x_2)\in\Sigma_e$ we have $f(x_1)$ and $f(x_2)$ separated by $e$ in $T$. Let $m(x_1,x_2)$ denote the pair of integers $(d(f(x_1),e),d(f(x_2),e))$. There is a uniform bound on these integers, since there is a uniform bound on $d(f(x_1),f(x_2))$ for $(x_1,x_2)$ ranging over all edges in $X$. Moreover, if $g\in G$ and $(x_1,x_2),(gx_1,gx_2)\in\Sigma_e$ with $m(x_1,x_2)=m(gx_1,gx_2)$, then $g\in G_e$. As a result, $\Sigma_e$ is a union of finitely many $G_e$-orbits of edges.
	Since $X$ is locally finite, it follows that $G_e$ acts cocompactly on $X_1\cap X_2$. As $G_e$ is finitely generated, it is also easy to see that $X_1\cap X_2$ is connected for large enough $L$.
	
	Now add 2-cells to $X$ according to a finite presentation for $G$. This makes $X$ simply connected, and $G=\pi_1(X/G)$. Furthermore, for large enough $L$, each 2-cell will have attaching map contained entirely in either $X_1$ or $X_2$, so we may add each 2-cell to either $X_1$ or $X_2$ without increasing their 1-skeleta. After adding 2-cells, we still have $X=X_1\cup X_2$, and $X_1$ and $X_2$ are still one-ended.
	
	By hypothesis, $G_e$ has more than one end, hence there is a finite subcomplex $K$ of $X$ such that $(X_1\cap X_2)-K$ has more than one unbounded component.
	We apply Theorem \ref{GeoSplit} to deduce that $X$, and hence $G$, does not have pro-finite first homology at infinity. Finally, Theorem \ref{GMpro} implies that $H^2(G,\mathbb ZG)\ne \{0\}$.
\end{proof}

\section{An example of a splitting obtained from folding}\label{sec:artificial}

We give an example of a finitely presented one-ended simply connected at infinity group that splits non-trivially over a finitely generated infinite-ended group. 
So there is no result that implies a general non-trivial splitting of a one-ended finitely presented group over an infinite-ended group is not simply connected at infinity.
Our example is obtained via \emph{folding}, an operation originally introduced for graphs by Stallings \cite{Stallings83}, and applied to group splittings by Bestvina and Feighn \cite{BestvinaFeighn91}.
More precisely, our example is obtained via the following lemma.

\begin{lemma}\label{Art} 
Suppose $G$ splits as $A\ast_CB$ and $D$ is a subgroup of $B$ containing $C$. Then $G$ also splits as $(A\ast_CD)\ast_D B$.
\end{lemma}
\begin{proof}
	We can transform the Bass--Serre tree of the first splitting into the Bass--Serre tree of the second splitting by  folding together the edges that correspond to cosets of $C$ which lie in the same coset of $D$.
If $D$ is finitely generated, then this can be expressed as a finite sequence of folds in the sense of Bestvina--Feighn \cite{BestvinaFeighn91}.

One can also argue using group presentations as follows.
Let $\mathcal P_Q=\langle S_Q:R_Q\rangle$ be a presentation for $Q$, where $Q\in\{A, B, C, D\}$. Assume that $\mathcal P_C$ is a sub-presentation of $\mathcal P_A$ and $\mathcal P_D$ (so that $S_C\subset S_A$, $R_C\subset R_A$, $S_C\subset S_D$ and $R_C\subset R_D$). Also assume that $\mathcal P_D$ is a sub-presentation of $\mathcal P_B$. Then a presentation for $G$ is
$$\mathcal P_G=\langle S_A\cup S_B: R_A\cup R_B\rangle.$$
A presentation for $(A\ast_CD)\ast_D B$ is 
$$\langle S_A\cup S_D\cup S_B:R_A\cup R_D\cup R_B\rangle.$$
This last presentation is $\mathcal P_G$ since $S_D\subset S_B$ and $R_D\subset R_B$. 
\end{proof}

\begin{example} \label{Ex1} 
Let $\mathbb Z_x$ be the infinite cyclic group with generator $x$. 
Let $A=\mathbb Z_a\times \mathbb Z_b\times \mathbb Z_f=\mathbb Z^3$,  $B=((\mathbb Z_a\times \mathbb Z_b)\ast \mathbb Z_c)\times \mathbb Z_d\times \mathbb Z_e$ and $C=\mathbb Z_a\times \mathbb Z_b$. The groups $A$ and $B$ are simply connected at infinity by \cite[Theorem 2]{J82a} and $C$ is one-ended. So $G=A\ast_C B$ is one-ended and simply connected at infinity by \cite[Theorem 2]{J82b}.
 
Let $D=(\mathbb Z_a\times \mathbb Z_b)\ast \mathbb Z_c$. As in Lemma \ref{Art}, the group $G$ has a second decomposition as 
$$G= (A\ast_CD)\ast_D B,$$
so the one-ended simply connected at infinity group $G$ splits non-trivially over the infinite-ended group $D=(\mathbb Z_a\times \mathbb Z_b)\ast \mathbb Z_c$. 
Note also that $B$ is one-ended, so the halfspace corresponding to $B$ in the above splitting is one-ended by Lemma \ref{lem:oneendedvertexgroup}, however, $D$ is infinite-ended, so the halfspace corresponding to $A\ast_CD$ has more than one end by Corollary \ref{Acyclic}.
\end{example}

\section{Some examples of splittings with one-ended halfspaces}\label{sec:examplesoneended}

The following proposition provides a general construction of HNN extensions with one-ended halfspaces (without appealing to more advanced theory like JSJ splittings).

\begin{proposition}\label{prop:stablecommute}
	Suppose $G=A\ast_C$, $A$ and $C$ are finitely generated and infinite, $t$ is the stable letter and $t^{-1}ct=c$ for all $c\in C$. If $G$ is one-ended then all the halfspaces in this splitting are one-ended.
	In particular, if $G$ is also finitely presented and $C$ has more than one end, then $H^2(G,\mathbb ZG)\ne \{0\}$ (by Corollary \ref{Acyclic}).
\end{proposition}
\begin{proof}
	We use the construction of halfspaces from Definition \ref{defn:halfspaceamal}.
	We have $S_A$ a finite generating set for $A$, and $S=S_A\cup\{t\}$ a finite generating set for $G$.
	The halfspace $\h^+$ is the induced subgraph of $\Gamma=\Cay(G,S)$ with vertex set corresponding to elements of $G$ with normal forms that do not start with an element of $C$ followed by a negative power of $t$.
	The opposite halfspace is denoted $\h^-$. If we choose $S$ so that it contains a generating set for $C$, then $\h^+\cap\h^-$ is connected and is isomorphic to a Cayley graph of $C$ (Remark \ref{remk:CayleyC}). We denote $\h^+\cap\h^-$ by $\Gamma_C$.
	We show $\h^+$ is one-ended (the argument for $\h^-$ is similar). Suppose a finite subgraph $K$ of $\Gamma$ separates $\h^+$, with at least two unbounded components $C_1$ and $C_2$. First observe that the intersection of $\Gamma_C$ with each of the components $C_1$ and $C_2$ is unbounded (otherwise $K$ union the intersection of this component with $\Gamma_C$ is a finite subgraph of $\Gamma$ separating $\Gamma$ with more than one unbounded component, but this implies $\Gamma$ has more than one end -- contrary to our hypothesis).  
	Now the product set $\{1,t,t^2,\dots\}C\subset G$ spans a subgraph $\Delta$ of $\h^+$, and $\Delta$ is a product of two infinite graphs (since $t$ commutes with every element of $C$), hence it is one-ended. But $C_1\cap\Delta$ and $C_2\cap\Delta$ are unbounded, hence they cannot be separated in $\Delta$ by the finite subgraph $K$, a contradiction.
\end{proof}

Next we show that a certain ``doubling" process always leads to amalgamations with one-ended halfspaces whenever the double is one-ended. If $A$ is a  group with subgroup $C$ then let $i:A\to A'$ be an isomorphism such that $i(C)=C'$. The {\it double} of $A$ across $C$ is the amalgamation $A\ast_{C=C'}A'$. 

\begin{proposition} \label{prop:double}
	Suppose $A$ is a finitely generated group with finitely generated subgroup $C$. Let $G=A\ast_{C=C'} A'$ be the double of $A$ across $C$. If $G$ is one-ended, then the halfspaces of $G$ with respect to $C$ are one-ended. In particular, if $G$ is also finitely presented and $C$ has more than one end, then $H^2(G,\mathbb ZG)\ne \{0\}$ (by Corollary \ref{Acyclic}).
\end{proposition}
\begin{proof} 
	First notice that there is a unique automorphism $r:G\to G$ satisfying $r(a)=i(a)$ for all $a\in A$ and $r(a')=i^{-1}(a')$ for all $a'\in A'$. If $X$ is the Cayley graph of $G$ with respect to a symmetric set of generators for $A$ and $A'$ (and $C$ and $C'$), then $r$ induces an involution $\hat r: X\to X$ that fixes the sub-Cayley graph for $C=C'$ and exchanges halfspaces. More precisely, if $\h=\h_A$ and $\h'=\h_{A'}$ are the halfspaces of $X$ given by Definition \ref{defn:halfspaceamal}, then $\hat r(\h)= \h'$ and $\hat r(\h')=\h$. Suppose $K$ is a compact subset of $\h$. Since $G$ (and hence $X$) is one-ended, there is a compact set $D\subset X$ such that any two points  $x,y \in X-D$ can be joined by a path in $X-(K\cup \hat r(K))$. Suppose $x$ and $y$ are points in $\h-D$ and $\alpha$ is a path in $X-(K\cup \hat r(K))$ from $x$ to $y$. Decompose $\alpha$ as $(\alpha_1, \alpha_1', \ldots, \alpha_n,\alpha_n')$ where $\alpha_i$ (respectively $\alpha_i'$) has image in $\h$ (respectively $\h'$). The path $(\alpha_1, \hat r(\alpha_1'), \ldots, \alpha_n, \hat r (\alpha_n'))$ is a path from $x$ to $y$ in $\h-K$. Hence $\h$ is one-ended (and $\h'$ similarly).
\end{proof}

The previous two propositions are good ways to produce examples of groups $G$ such that $H^2(G,\mathbb ZG)\ne \{0\}$, and hence groups that are not duality groups.
We provide two such examples below (Examples \ref{exm:HNN} and \ref{exm:amalgam}).
Recall that a group $G$ that acts freely and cocompactly on a contractible cell complex
$X$ (by permuting cells) is a {\it duality group of dimension} $n$ if $H^i (G; \mathbb Z G)$ vanishes when $i\ne n$ and is torsion-free (as an abelian group) when $i = n$ (see R. Bieri and B. Eckmann's article \cite{BE73}). 

Examples \ref{exm:HNN} and \ref{exm:amalgam} are chosen in such a way that we cannot see a method other than using one-endedness of halfspaces to show that they are not duality groups.
Before giving the examples, we discuss two other methods for showing that a group is not a duality group (which fail for our examples).

\textbf{(Co)homological dimension:}
Any nontrivial amalgam $G=A*_C B$ or HNN extension $G=A*_C$ of an $n$-dimensional duality group $G$ requires $n-1\leq \mathrm{cd}(C)\leq\mathrm{cd}(A),\mathrm{cd}(B)\leq n$ \cite[Proposition 9.14]{Bieri81}.
Furthermore, if $G$ is a Poincar\'e duality group, then every subgroup of $G$ either has finite index or has homology dimension at most $n-1$ \cite[Proposition 9.22]{Bieri81}.

\textbf{Mayer--Vietoris sequences:}
The Mayer--Vietoris long exact sequence can be applied to both amalgams and HNN extensions of groups. We state the result for amalgams.

\begin{theorem} [Theorem 2.10, \cite{Bieri76}] \label{AmalH2} Suppose the group $G$ decomposes as $G = A \ast_C B$, where $A$ and $B$ are finitely presented and $C$ is finitely generated.  Then one has a natural long exact sequence:
	$$\cdots \to H^k(G,\mathbb ZG)\to H^k(A,\mathbb ZG)\oplus H^k(B,\mathbb ZG)\to H^k(C, \mathbb ZG)\to H^{k+1}(G,\mathbb ZG)\to\cdots$$
\end{theorem}

Suppose that the group $G$ decomposes as $G=A\ast_CB$, where $A$ and $B$ are one-ended and finitely presented and $C$ is finitely generated with more than one end. The number of ends of a finitely generated group is equal to $rank (H^1(G, RG))+1$ (see Theorem 13.5.5, \cite{Geoghegan08}). So we have $H^1(A,\mathbb ZG )=H^1(B,\mathbb ZG)=\{0\}$. 
Applying the above exact sequence, we obtain: 
$$0\to H^1(C,\mathbb ZG)\to H^2(G,\mathbb ZG)$$ 
In particular, $H^2(G,\mathbb ZG)\ne \{0\}$ (which implies Jackson's theorem). However, when $A$ or $B$ is not one-ended, this sequence gives little information about $H^2(G,\mathbb ZG)$.

\begin{example}\label{exm:HNN}
Let $A=\mathbb Z^n$ ($n\geq2$) with basis $\{a_1,\ldots, a_n\}$. Consider the free product $A\ast\mathbb{Z}$, where the second factor is generated by an element $s$. Let $C$ be the subgroup of $A\ast \langle s\rangle$ with generating set $\{a_1,\ldots, a_{n-1}, s^2, sa_1sa_2\cdots sa_ns\}$. Note that $C\cong\mathbb Z^{n-1}\ast F_2$. Define
$$G=(A\ast \langle s\rangle)\ast_C$$
with stable letter $t$ such that $t^{-1}ct=c$ for all $c\in C$. 

Note that the base group $A\ast \langle s\rangle$ has cohomological dimension $n$ and the edge group $C$ has cohomological dimension $n-1$, so $G$ looks like a duality group from the perspective of cohomological dimension.
The point of including the element $as_1sa_2\cdots sa_ns$ in $C$ is to prevent other splittings of $G$ over groups of lower cohomological dimension.

Now we show that $G$ is one-ended. Assume not. As $G$ is torsion free, it splits non-trivially as $U\ast V$ with $U$ and $V$ non-trivial torsion free groups. The subgroup generated by $C$ and the stable letter $t$ has the form $C\times \langle t\rangle$ and so is  one-ended. Its intersection with the one-ended group $A$ is $\langle a_1,\ldots, a_{n-1}\rangle$, which is infinite. This implies the group generated by $A$, $t$, $s^2$ and $as_1sa_2\cdots sa_ns$ is one-ended and so must be contained in a conjugate of $U$ or $V$ (say $U$). Killing the normal closure of $U$ in $G$ leaves $V$. But killing $U$ also kills $A$, $t$ and $s^2$ and so $V$ is at most $\mathbb Z_2$, a contradiction.  

The edge group $C$ is infinite-ended and $G$ is finitely presented, so $H^2(G,\mathbb ZG)\ne \{0\}$ by Proposition \ref{prop:stablecommute}. Hence $G$ is not a duality group.
\end{example}

\begin{example}\label{exm:amalgam}
Consider the group 
$$A=(F_2\times F_2\times F_2)\ast F_2=(\langle x_1,x_2\rangle\times \langle y_1,y_2\rangle\times \langle z_1,z_2\rangle)\ast \langle a_1,a_2\rangle.$$
Let $G=A\ast _{C=C'}  A'$ be the double of $A$ across $C$, where 
$$C =\langle x_1,y_1,z_1, a_1a_2a_1^{-1}a_2^{-1}, a_1x_2 a_2y_2 a_1z_2\rangle.$$ 
Note that $A$, $A'$ and $C\cong\mathbb{Z}^3\ast F_2$ have cohomological dimension 3.
Also note that $A$ and $A'$ are infinite-ended, so we cannot use the Mayer--Vietoris sequence to show that $H^2(G,\mathbb ZG)\ne \{0\}$.

Now we show $G$ is one-ended. Suppose not. As $G$ is torsion free, it splits non-trivially as $U\ast V$. The subgroup of $G$ generated by $$\{x_1=x'_1,x_2,x'_2,y_1=y'_1,y_2,y'_2,z_1=z'_1,z_2,z'_2\}$$ is isomorphic to $F_2^3 *_{\Z^3}F_2^3$, so is one-ended (where $x'_i,y'_i,z'_i$ are the images of $x_i,y_i,z_i$ under the isomorphism $A\to A'$), as is the subgroup generated by $$\{a_1,a_2,a_1',a_2'\}$$ (a surface group). These subgroups generate $G$. One lies in a conjugate $gUg^{-1}$, the other in a conjugate of $hVh^{-1}$ (killing the normal closure of $U$ leaves $V$). By a standard ping-pong argument in the Bass--Serre tree of the free product $U\ast V$, the subgroup $\langle gUg^{-1},hVh^{-1}\rangle$ splits as a free product $gUg^{-1}\ast hVh^{-1}$.
But then the relator $a_1x_2 a_2y_2 a_1z_2(a_1'x_2' a_2'y_2' a_1'z_2')^{-1}$ has length 11 with respect to the normal form in $gUg^{-1}\ast hVh^{-1}$, which is impossible. Instead $G$ is one-ended. 

The edge group $C$ is infinite-ended and $G$ is finitely presented, so Proposition \ref{prop:double} implies that the halfspaces of $G=A\ast _{C=C'}  A'$ are one-ended, and that $H^2(G,\mathbb ZG)\ne \{0\}$. In particular $G$ is not a duality group. 
\end{example}

\bibliographystyle{amsalpha}
\bibliography{ref}{}

\end{document}